\documentclass[12pt]{amsart}
\usepackage{amssymb}
\usepackage{amsmath}
\usepackage{xspace}
\usepackage[all]{xypic}
\usepackage{amscd}
\usepackage{enumerate}
\usepackage{color}
\usepackage{graphicx}
\usepackage{mathtools}
\theoremstyle{plain}
\newtheorem{theorem}{Theorem}[section]
\newtheorem{corollary}[theorem]{Corollary}
\newtheorem{proposition}[theorem]{Proposition}
\newtheorem{lemma}[theorem]{Lemma}
{\theoremstyle{remark}

\newtheorem{remark}[theorem]{Remark}}
{\theoremstyle{definition}
\newtheorem{definition}[theorem]{Definition}

}
\numberwithin{equation}{section}
\numberwithin{table}{section}
\numberwithin{figure}{section}

\newcommand{\rom}{\renewcommand{\labelenumi}{{\rm (\roman{enumi})}}%
\renewcommand{\itemsep}{0pt}}

\newcommand{\Z}{\mathbb{Z}}

\newcommand{\R}{\mathbb{R}}
\newcommand{\C}{\mathbb{C}}

\newcommand{\e}{\varepsilon}

\newcommand{\fS}{\mathfrak{S}}
\newcommand{\ha}{\widehat{\alpha}}
\newcommand{\tF}{\widetilde{F}}
\newcommand{\tO}{\widetilde{O}}
\newcommand{\bU}{\mathbb{U}}
\newcommand{\hT}{\widehat{T}}

\newcommand{\mo}{\makebox[0.8em]{$-1$}}
\newcommand{\Ca}{$C^*$-al\-ge\-bra }

\newcommand{\Csa}{$C^*$-sub\-al\-ge\-bra }
\newcommand{\CsA}{$C^*$-sub\-al\-ge\-bra}
\newcommand{\shom}{$*$-ho\-mo\-mor\-phism }

\DeclareMathOperator{\id}{id}

\DeclareMathOperator{\Ad}{Ad}
\DeclareMathOperator{\coker}{coker}

\newcommand{\Conv}{\mathop{\scalebox{1.3}{\raisebox{-0.2ex}{$\ast$}}}}
\begin{document}

\title{On the magic square C*-algebra of size 4}

\author[Takeshi KATSURA]{Takeshi KATSURA}
\address{Department of Mathematics\\
Faculty of Science and Technology\\
Keio University\\
3-14-1 Hiyoshi, Kouhoku-ku, Yokohama\\
223-8522 JAPAN}
\email{katsura@math.keio.ac.jp}

\author[Masahito OGAWA]{Masahito OGAWA}
\address{Library \& Information center\\
Yokohama City University\\
22-2 Seto, Kanazawa-ku, Yokohama\\
236-0027 JAPAN}

\author[Airi TAKEUCHI]{Airi TAKEUCHI}
\address{Karlsruhe Institute of Technology\\
Department of Mathematics\\
76128 Karlsruhe, Germany}
\email{airi.takeuchi@partner.kit.edu}



\begin{abstract}
In this paper, we investigate the structure of 
the magic square C*-algebra $A(4)$ of size 4. 
We show that a certain twisted crossed product of $A(4)$ 
is isomorphic to the homogeneous C*-algebra $M_4(C(\R P^3))$. 
Using this result, we show that $A(4)$ is 
isomorphic to the fixed point algebra of $M_4(C(\R P^3))$ 
by a certain action. 
From this concrete realization of $A(4)$, 
we compute the K-groups of $A(4)$ and their generators. 
\end{abstract}

\maketitle

\setcounter{section}{-1}

\section{Introduction}

Let $n=1,2,\ldots$. 
The magic square C*-algebra $A(n)$ of size $n$ is 
the underlying C*-algebra of the quantum group $A_s(n)$ 
defined by Wang in \cite{W} as a free analogue of 
the symmetric group $\fS_n$. 
In \cite[Proposition~1.1]{BC}, 
it is claimed that for $n=1,2,3$, 
$A(n)$ is isomorphic to $\C^{n!}$, 
and hence commutative and finite dimensional. 
We give the proof of this fact in Proposition~\ref{Prop:123}. 
In \cite[Proposition~1.2]{BM} 
it is proved that for $n \geq 4$ 
$A(n)$ is non-commutative and infinite dimensional. 
We see that 
for $n \geq 5$ $A(n)$ is not exact (Proposition~\ref{Prop:nonexact}). 
Something interesting happens for $A(4)$ (see \cite{BM,BB,BC}). 
In \cite{BM}, 
Banica and Moroianu constructed a \shom from $A(4)$ to $M_4(C(SU(2)))$ 
by using the Pauli matrices, 
and showed that it is faithful in some weak sense. 
In \cite{BC}, 
Banica and Collins showed that the \shom above is in fact faithful 
by using integration techniques. 
We reprove this fact in Corollary~\ref{Cor:faithful}. 
Our method uses a twisted crossed product. 
The following is the first main result. 

\medskip
\noindent
{\bfseries Theorem A} (Theorem~\ref{MainThm1}){\bfseries .} 
The twisted crossed product $A(4) \rtimes_{\alpha}^{\text{tw}} (K \times K)$ 
is isomorphic to $M_4(C(\R P^3))$. 
\medskip

The notation in this theorem is explained in Section~\ref{Sec:twist}. 
From this theorem, 
we see that the magic square C*-algebra $A(4)$ of size $4$ 
is isomorphic to a \Csa of the homogeneous \Ca $M_4(C(\R P^3))$. 
The next theorem, which is the second main result, expresses 
this \Csa as a fixed point algebra of $M_4(C(\R P^3))$. 

\medskip
\noindent
{\bfseries Theorem B} (Theorem~\ref{MainThm2}){\bfseries .} 
The fixed point algebra $M_4(C(\R P^3))^{\beta}$ of the action $\beta$ is 
isomorphic to $A(4)$. 
\medskip

See Section~\ref{Sec:Act} for the definition of the action $\beta$. 
Since $\beta$ is concrete, 
we can analyze $M_4(C(\R P^3))^{\beta}$ very explicitly. 
In particular, we can compute the K-groups of $M_4(C(\R P^3))^{\beta}$ 
explicitly. 
As a corollary we get the following which is the third main result. 

%\newpage

\medskip
\noindent
{\bfseries Theorem C} (Theorem~\ref{MainThm3}){\bfseries .} 
We have $K_0(A(4)) \cong \Z^{10}$ and $K_1(A(4)) \cong \Z$.
More specifically, $K_0(A(4))$ is generated by $\{[p_{i,j}]_0\}_{i,j=1}^4$, 
and $K_1(A(4))$ is generated by $[u]_1$.

The positive cone $K_0(A(4))_+$ of $K_0(A(4))$ is generated 
by $\{[p_{i,j}]_0\}_{i,j=1}^4$ as a monoid. 
\medskip

Note that $\{p_{i,j}\}_{i,j=1}^4$ is the generating set of $A(4)$ 
consisting of projections, and $u$ is the defining unitary 
(see Definition~\ref{Def:defuni}). 
We should remark that the computation 
$K_0(A(4)) \cong \Z^{10}$ and $K_1(A(4)) \cong \Z$ 
and that $K_0(A(4))$ is generated by $\{[p_{i,j}]_0\}_{i,j=1}^4$
were already obtained by Voigt in \cite{V} 
by using Baum-Connes conjecture for quantum groups. 
In fact, Voigt got the corresponding results for $A(n)$ with $n \geq 4$. 
Theorem~C gives totally different proofs for the results by Voigt in \cite{V} 
by analyzing the structure of $A(4)$ directly 
which seems not to be applied to $A(n)$ for $n > 4$. 
That $K_1(A(4))$ is generated by $[u]_1$ was not obtained in \cite{V}, 
and is a new result. 
Combining this result with 
the computation that $K_1(A(n)) \cong \Z$ for $n \geq 4$ in \cite{V}
and the easy fact that 
the surjection $A(n) \to A(4)$ in Corollary~\ref{Cor:ntom} for $n \geq 4$ 
sends the defining unitary to the direct sum of the defining unitary 
and the units, 
we obtain that $K_1(A(n)) \cong \Z$ is generated by the $K_1$ class 
of the defining unitary for $n \geq 4$. 
We would like to thank Christian Voigt for the discussion 
about this observation. 

This paper is organized as follows. 
In Section~\ref{Sec:Def},
we define magic square C*-algebras $A(n)$ 
and their abelianizations $A^{\text{ab}}(n)$. 
In Section~\ref{Sec:general}, 
we investigate $A(n)$ for $n \neq 4$. 
From Section~\ref{Sec:twist}, we study $A(4)$. 
In Section~\ref{Sec:twist}, 
we introduce the twisted crossed product 
$A(4) \rtimes_{\alpha}^{\text{tw}} (K \times K)$, 
and state Theorem~A. 
We give the proof of Theorem~A 
from Section~\ref{Sec:RP3} to Section~\ref{Sec:Inv}. 
In Section~\ref{Sec:Act}, 
we state and prove Theorem~B. 
From Section~\ref{Sec:quotient} to Section~\ref{Sec:KA}, 
we prove Theorem~C. 

\medskip

\noindent
{\bfseries Acknowledgments.} 
The first author thank Junko Muramatsu 
for helping the research in the beginning of the research. 
The authors are grateful to Makoto Yamashita 
for calling attention to \cite{V}, 
and to Christian Voigt for the discussion on the results in \cite{V}. 
The first author was supported by JSPS KAKENHI Grant Number JP18K03345.

\section{Definitions of and basic facts on magic square C*-algebras}
\label{Sec:Def} 

\begin{definition}\label{Def:magic}
Let $n=1,2,\ldots$. 
The {\em magic square C*-algebra of size $n$} is 
the universal unital C*-algebra $A(n)$ generated by $n \times n$ projections 
$\{p_{i,j}\}_{i,j=1}^n$ satisfying 
\begin{align*}
\sum_{i=1}^n p_{i,j} =1 
&\quad (j=1, 2, \ldots , n), &%\nonumber \\
\sum_{j=1}^n p_{i,j} =1 
&\quad (i=1, 2, \ldots , n). 
\end{align*}
\end{definition}

\begin{remark}
The magic square C*-algebra $A(n)$ is 
the underlying C*-algebra of the quantum group $A_s(n)$ 
defined by Wang in \cite{W} as a free analogue of 
the symmetric group $\fS_n$. 
\end{remark}

We fix a positive integer $n$. 
Let $\fS_n$ be the symmetric group of degree $n$ 
whose element is considered to be a bijection on the set 
$\{1,2,\ldots, n\}$. 

\begin{definition}\label{Def:symm}
By the universality of $A(n)$, 
there exists an action 
$\alpha \colon \fS_n \times \fS_n 
\curvearrowright A(n)$ 
defined by 
\[
\alpha_{(\sigma,\mu)}(p_{i,j})=p_{\sigma(i),\mu(j)}
\]
for $(\sigma,\mu) \in \fS_n \times \fS_n$ 
and $i,j =1, 2, \ldots , n$. 
\end{definition}

\begin{definition}
Let $A^{\text{ab}}(n)$ be 
the universal unital C*-algebra generated by $n \times n$ projections 
$\{p_{i,j}\}_{i,j=1}^n$ satisfying the relations in Definition~\ref{Def:magic} 
and 
\[
p_{i,j}p_{k,l}=p_{k,l}p_{i,j} \qquad (i,j,k,l =1, 2, \ldots , n).
\]
\end{definition}

The following lemma follows immediately from the definitions. 

\begin{lemma}
The C*-algebra $A^{\text{ab}}(n)$ is the abelianization of $A(n)$. 
More specifically, there exists a natural surjection 
$A(n) \twoheadrightarrow A^{\text{ab}}(n)$ 
sending each projection $p_{i,j}$ to $p_{i,j}$, 
and every $*$-homomorphism from $A(n)$ to an abelian C*-algebra 
factors through this surjection. 
\end{lemma}

\begin{proposition}\label{Prop:Aab}
The abelian C*-algebra $A^{\text{ab}}(n)$ is 
isomorphic to the C*-algebra $C(\fS_n)$ 
of continuous functions on the discrete set $\fS_n$. 
\end{proposition}

\begin{proof}
For each $\sigma \in \fS_n$, 
we define a character $\chi_\sigma$ of $A^{\text{ab}}(n)$ by 
\[
\chi_\sigma(p_{i,j})=
\begin{cases}
1 & (i = \sigma(j))\\
0 & (i \neq \sigma(j)).
\end{cases}
\]
Note that such a character $\chi_\sigma$ uniquely exists 
by the universality of $A^{\text{ab}}(n)$. 
It is easy to see that any character of $A^{\text{ab}}(n)$ 
is in the form of $\chi_\sigma$ for some $\sigma \in \fS_n$. 
This shows that $A^{\text{ab}}(n)$ is isomorphic to $C(\fS_n)$ 
by the Gelfand theorem. 
\end{proof}

We can compute minimal projections of $A^{\text{ab}}(n)$ 
as follows.

\begin{proposition}\label{Prop:Aabp}
For $\sigma \in \fS_n$, we set 
\[
p_\sigma \coloneqq p_{\sigma(1),1}p_{\sigma(2),2}\cdots p_{\sigma(n),n} 
\in A^{\text{ab}}(n).
\]
Then $\{p_\sigma\}_{\sigma \in \fS_n}$ is the set of 
minimal projections of $A^{\text{ab}}(n)$. 
\end{proposition}

\begin{proof}
Since $A^{\text{ab}}(n)$ is commutative, 
$p_\sigma$ is a projection for every $\sigma \in \fS_n$. 
For $\sigma \in \fS_n$, let $\chi_\sigma$ be the character defined 
in the proof of Proposition~\ref{Prop:Aab}. 
Then we have 
\[
\chi_{\sigma'}(p_\sigma)=
\begin{cases}
1 & (\sigma' = \sigma)\\
0 & (\sigma' \neq \sigma)
\end{cases}
\]
for $\sigma,\sigma' \in \fS_n$. 
This shows that $\{p_\sigma\}_{\sigma \in \fS_n}$ is the set of 
minimal projections of $A^{\text{ab}}(n)$. 
\end{proof}

For each $\sigma \in \fS_n$, 
we can define a character $\chi_\sigma$ of $A(n)$ 
by the same formula as in the proof of Proposition~\ref{Prop:Aab} 
(or to be the composition of the character $\chi_\sigma$ 
in the proof of Proposition~\ref{Prop:Aab} and the natural surjection 
$A(n) \twoheadrightarrow A^{\text{ab}}(n)$). 
With these characters we have the following 
as a corollary of of Proposition~\ref{Prop:Aab} 
(It is easy to show it directly). 

\begin{corollary}
The set of all characters of the magic square C*-algebra $A(n)$ 
is $\{\chi_\sigma\mid \sigma \in \fS_n\}$ 
whose cardinality is $n!$. 
\end{corollary}

\section{General results on magic square C*-algebras}\label{Sec:general}

\begin{proposition}\label{Prop:123}
For $n=1,2,3$, 
$A(n)$ is commutative. 
Hence the surjection $A(n)\twoheadrightarrow A^{\text{ab}}(n)$ 
is an isomorphism for $n=1,2,3$.
\end{proposition}

\begin{proof}
For $n=1$ and $n=2$, 
it is easy to see $A(1)\cong \C$ and $A(2)\cong \C^2$. 
To show that $A(3)$ is commutative, 
it suffices to show $p_{1,1}$ commutes with $p_{2,2}$. 
In fact if $p_{1,1}$ commutes with $p_{2,2}$, 
we can see that $p_{1,1}$ commutes with $p_{2,3}$, $p_{3,2}$ and $p_{3,3}$ 
using the action $\alpha$ defined in Definition~\ref{Def:symm}. 
Then $p_{1,1}$ commutes with every generators 
because $p_{1,1}$ is orthogonal to and hence commutes with 
$p_{1,2}$, $p_{1,3}$, $p_{2,1}$ and $p_{3,1}$. 
Using the action $\alpha$ again, 
we see that every generators 
commutes with every generators. 

Now we are going to show that $p_{1,1}$ commutes with $p_{2,2}$. 
We have 
\begin{align*}
p_{1,1}p_{2,2}
=(1-p_{1,2}-p_{1,3})p_{2,2}
&=p_{2,2}-p_{1,3}p_{2,2}\\
&=p_{2,2}-(1-p_{2,3}-p_{3,3})p_{2,2}
=p_{3,3}p_{2,2}.
\end{align*}
By symmetry, we have $p_{2,2}p_{3,3}=p_{1,1}p_{3,3}$ 
and $p_{3,3}p_{1,1}=p_{2,2}p_{1,1}$. 
Hence we get
\begin{align*}
p_{1,1}p_{2,2}=p_{3,3}p_{2,2}=(p_{2,2}p_{3,3})^*
=(p_{1,1}p_{3,3})^*=p_{3,3}p_{1,1}=p_{2,2}p_{1,1}.
\end{align*}
This completes the proof.
\end{proof}

\begin{proposition}\label{Prop:freeprod}
Let $n_1,n_2,\ldots,n_k$ be positive integers, 
and set $n=\sum_{j=1}^kn_j$. 
There exists a surjection from $A(n)$ to 
the unital free product $\Conv_{j=1}^kA(n_j)$. 
\end{proposition}

\begin{proof}
The desired surjection is obtained 
by sending the generators $\{p_{i,j}\}_{i,j=1}^{n_1}$ of $A(n)$ 
to the generators of $A(n_1) \subset \Conv_{j=1}^kA(n_j)$, 
the generators $\{p_{i,j}\}_{i,j=n_1+1}^{n_1+n_2}$ of $A(n)$ 
to the generators of $A(n_2) \subset \Conv_{j=1}^kA(n_j)$ 
and so on, and by sending the other generators of $A(n)$ to $0$. 
\end{proof}

\begin{corollary}\label{Cor:n+1ton}
Let $n$ be a positive integer. 
There exists a surjection from $A(n+1)$ to $A(n)$. 
\end{corollary}

\begin{proof}
This follows from Proposition~\ref{Prop:freeprod} 
because $A(n) * A(1) \cong A(n) * \C \cong A(n)$. 
\end{proof}

\begin{corollary}\label{Cor:ntom}
Let $n,m $ be positive integers with $n \geq m$. 
There exists a surjection from $A(n)$ to $A(m)$. 
\end{corollary}

\begin{proof}
This follows from Corollary~\ref{Cor:n+1ton}. 
\end{proof}

\begin{proposition}\label{Prop:nonexact}
For $n\geq 5$, 
$A(n)$ is not exact. 
\end{proposition}

\begin{proof}
Note that an image of an exact C*-algebra is exact 
(see \cite[Corollary~9.4.3]{BO}). 
By Corollary~\ref{Cor:ntom}, 
it suffices to show that $A(5)$ is not exact. 
By Proposition~\ref{Prop:freeprod}, 
there exists a surjection from $A(5)$ to 
$A(2)*A(3) \cong \C^2 * \C^6$ which is not exact 
(see \cite[Proposition~3.7.11]{BO}). 
This completes the proof. 
\end{proof}

The \Ca $A(4)$ is not commutative, but is exact, 
in fact is subhomogeneous 
(Corollary~\ref{Cor:faithful}). 
From the next section, 
we investigate the structure of $A(4)$.

\section{Twisted crossed product}\label{Sec:twist}

We denote elements $\sigma \in \fS_4$ by 
$(\sigma(1)\sigma(2)\sigma(3)\sigma(4))$. 
We define the Klein (four) group $K$ by 
\[
K \coloneqq \{ t_1, t_2, t_3, t_4\} \subset \fS_4
\]
where $t_1$ is the identity $(1234)$ of $\fS_4$, 
$t_2=(2143)$, $t_3=(3412)$ and $t_4=(4321)$. 
The group $K$ is isomorphic to $(\Z/2\Z)\times (\Z/2\Z)$. 

We choose the indices so that we have 
$t_it_j=t_{t_i(j)}$ for $i,j=1,2,3,4$. 
Note that we have $t_i(j)=t_j(i)$ for $i,j=1,2,3,4$.

\begin{definition}
Define unitaries $c_1,c_2,c_3,c_4$ in $M_2(\C)$ by 
\begin{equation*}
c_1 \coloneqq
\begin{pmatrix}
1 & 0 \\
0 & 1 
\end{pmatrix}, 
\quad
c_2 \coloneqq
\begin{pmatrix}
\sqrt{-1} & 0 \\
0 & -\sqrt{-1}
\end{pmatrix}, 
\quad
c_3 \coloneqq
\begin{pmatrix}
0 & 1 \\
-1 & 0 
\end{pmatrix}, 
\quad
c_4 \coloneqq
\begin{pmatrix}
0 & \sqrt{-1} \\
\sqrt{-1} & 0 
\end{pmatrix}.
\end{equation*}
\end{definition}

The unitaries $c_1,c_2,c_3,c_4$ are called 
the Pauli matrices. 

\begin{definition}
Put $\omega=(1342) \in \mathfrak{S}_4$.
Define the map $\varepsilon\colon \{1,2,3,4\}^2 \to \{1,-1\}$ by
\begin{equation*}
\varepsilon(i,j) \coloneqq 
\begin{cases}
1 & \text{($i=1$ or $j=1$ or $\omega(i)=j$)} \\
-1 & \text{(otherwise)},
\end{cases}
\end{equation*}
for each $i,j = 1,2,3,4$.
\end{definition}

\begin{table}[htb]
\begin{center}
\begin{tabular}{|c||c|c|c|c|} \hline
$i\diagdown j$ & $1$ &$2$  & $3$ & $4$ \\ \hline \hline
$1$& 1 & 1 & 1 & 1  \\ \hline
$2$& 1 & $-1$ & 1 & $-1$ \\ \hline
$3$& 1 & $-1$ & $-1$ & 1 \\ \hline
$4$& \,1\, & 1 & $-1$& $-1$ \\ \hline
\end{tabular}
\end{center}
\caption{Values of $\varepsilon(i,j)$}
\label{table:cocycle}
\end{table} 

We have the following calculation which can be proved straightforwardly.

\begin{lemma}\label{lem:pauli}
For $i,j=1,2,3,4$, 
we have $c_i c_j = \varepsilon(i,j) c_{t_i (j)}$.
\end{lemma}

From this lemma and the computation $t_it_j=t_{t_i(j)}$, 
we have the following lemma 
which means that 
$K^2 \ni (t_i,t_j)\mapsto \varepsilon(i,j) \in \{1,-1\}$ 
becomes a cocycle of $K$. 

\begin{lemma}\label{lem:cocycle}
For $i,j,k=1,2,3,4$, 
we have $\varepsilon(i, j) \varepsilon(t_i(j), k) 
= \varepsilon(i, t_j(k)) \varepsilon(j, k)$.
\end{lemma}

\begin{proof}
Compute $c_i c_j c_k$ in the two ways, 
namely $(c_i c_j) c_k$ and $c_i (c_j c_k)$. 
\end{proof}

Hence the following definition makes sense. 
Let us denote by the same symbol $\alpha$ 
the restriction of the action 
$\alpha\colon \fS_4\times \fS_4\curvearrowright A(4)$
to $K\times K \subset \fS_4\times \fS_4$.

\begin{definition}
Let $A(4) \rtimes_{\alpha}^{\text{tw}} (K \times K)$ 
be the twisted crossed product of the action $\alpha$ and 
the cocycle 
\[
(K \times K)^2\ni ((t_i,t_j),(t_k,t_l))\mapsto 
\varepsilon(i,k)\varepsilon(j,l) \in \{1,-1\}. 
\]
\end{definition}

By definition, 
$A(4) \rtimes_{\alpha}^{\text{tw}} (K \times K)$ 
is the universal $C^*$-algebra generated by the unital subalgebra $A(4)$ and 
unitaries $\{u_{i,j}\}_{i,j=1}^4$ such that 
\begin{align*}
u_{i,j}xu_{i,j}^*=\alpha_{(t_i,t_j)}(x)\qquad 
\text{for all $i,j$ and all $x \in A(4)$} 
\end{align*}
and 
\begin{align*}
u_{i,j}u_{k,l}=\varepsilon(i,k)\varepsilon(j,l)u_{t_i(k),t_j(l)}
\quad \text{for all $i,j,k,l$.} 
\end{align*}
We denote by $\mathcal{R}_{\text{u}}$ the latter relation. 
The former relation is equivalent to the relation
\begin{align*}
u_{i,j}p_{k,l}=p_{t_i(k),t_j(l)}u_{i,j}\qquad 
\text{for all $i,j,k,l$} 
\end{align*}
which is denoted by $\mathcal{R}_{\text{up}}$. 

Recall that $A(4)$ is the universal unital \Ca 
generated by the set $\{p_{i,j}\}_{i,j=1}^4$ of projections 
satisfying the following relation 
denoted by $\mathcal{R}_{\text{p}}$
\begin{align*}
\sum_{i=1}^4 p_{i,j} =1 
&\quad (j=1, 2, 3 , 4), &%\nonumber \\
\sum_{j=1}^4 p_{i,j} =1 
&\quad (i=1, 2, 3 , 4). 
\end{align*}

The following is the first main theorem. 

\begin{theorem}\label{MainThm1}
The twisted crossed product $A(4) \rtimes_{\alpha}^{\text{tw}} (K \times K)$ 
is isomorphic to $M_4(C(\R P^3))$. 
\end{theorem}

We finish the proof of this theorem in the end of Section~\ref{Sec:Inv}. 

To prove this theorem, 
we start with finite presentation of the \Ca $C(\R P^3)$ in the next section.

\section{Real projective space $\R P^3$}\label{Sec:RP3}

\begin{definition}
We set an equivalence relation $\sim$ on the manifold 
\[
S^3 \coloneqq \Big\{a=(a_1,a_2,a_3,a_4) \in \R^4\ \Big|\ \sum_{i=1}^4a_i^2=1\Big\}
\]
so that $a\sim b$ if and only if $a=b$ or $a=-b$. 
The quotient space $S^3/{\sim}$ is the real projective space $\R P^3$ 
of dimension 3. 
The equivalence class of $(a_1,a_2,a_3,a_4) \in S^3$ is denoted as 
$[a_1,a_2,a_3,a_4] \in \R P^3$. 
\end{definition}

\begin{definition}
For $i,j = 1,2,3,4$, we define a continuous function 
$f_{i,j}$ on $\R P^3$ by $f_{i,j}([a_1,a_2,a_3,a_4])=a_ia_j$ 
for $[a_1,a_2,a_3,a_4] \in \R P^3$. 
\end{definition}

Note that $f_{i,j}$ is a well-defined continuous function. 

\begin{lemma}\label{Lem:f}
The functions $\{f_{i,j}\}_{i,j=1}^4$ satisfy the following relation 
\begin{align*}
&f_{i,j}=f_{i,j}^*=f_{j,i}\quad \text{for all $i,j$,} \\
&f_{i,j}f_{k,l}=f_{i,k}f_{j,l} \quad \text{for all $i,j,k,l$,} \\
&\sum_{i=1}^4 f_{i,i}=1. 
\end{align*}
\end{lemma}

\begin{proof}
This follows from easy computation. 
\end{proof}

\begin{definition}
We denote by $\mathcal{R}_{\text{f}}$ 
the relation in Lemma~\ref{Lem:f}. 
\end{definition}

\begin{proposition}
The \Ca $C(\R P^3)$ is the universal unital \Ca 
generated by elements $\{f_{i,j}\}_{i,j=1}^4$ satisfying $\mathcal{R}_{\textrm{f}}$. 
\end{proposition}

\begin{proof}
Let $A$ be the universal unital \Ca 
generated by elements $\{f_{i,j}\}_{i,j=1}^4$ satisfying $\mathcal{R}_{\text{f}}$. 
For $i,j,k,l=1,2,3,4$, we have 
\begin{align*}
f_{i,j}f_{k,l}=f_{i,k}f_{j,l} =f_{k,i}f_{l,j} =f_{k,l}f_{i,j}.
\end{align*}
Hence $A$ is commutative. 
Thus there exists a compact set $X$ such that $A \cong C(X)$. 

By Lemma~\ref{Lem:f}, 
we have a unital \shom $A \to C(\R P^3)$. 
This induces a continuous map $\varphi \colon \R P^3 \to X$. 
It suffices to show that this continuous map is homeomorphic. 

We first show that $\varphi$ is injective. 
Take $[a_1,a_2,a_3,a_4]$ and $[b_1,b_2,b_3,b_4] \in \R P^3$ 
with $\varphi([a_1,a_2,a_3,a_4])=\varphi([b_1,b_2,b_3,b_4])$. 
Then, for $i,j=1,2,3,4$, we have $a_ia_j=b_ib_j$. 
Since $\sum_{i=1}^4a_i^2=1$, there exists $i_0$ 
such that $a_{i_0} \neq 0$. 
Set $\sigma = b_{i_0}/a_{i_0} \in \R$. 
Since $a_ia_{i_0}=b_ib_{i_0}$, 
we have $a_i=\sigma b_i$ for $i=1,2,3,4$. 
Since $\sum_{i=1}^4a_i^2=\sum_{i=1}^4b_i^2=1$, we get $\sigma=\pm 1$. 
Hence $[a_1,a_2,a_3,a_4] = [b_1,b_2,b_3,b_4]$. 
This shows that $\varphi$ is injective. 

Next we show that $\varphi$ is surjective. 
Take a unital character $\chi \colon A \to \C$ of $A$. 
To show that $\varphi$ is surjective, 
it suffices to find $[a_1,a_2,a_3,a_4] \in \R P^3$ 
such that $\chi(f_{i,j})=a_ia_j$ for all $i,j=1,2,3,4$. 
Since $\sum_{i=1}^4\chi(f_{i,i})=\chi\big(\sum_{i=1}^4f_{i,i}\big)=1$, 
there exists $i_0$ such that $\chi(f_{i_0,i_0})\neq 0$. 
Since 
\begin{align*}
f_{i_0,i_0} = f_{i_0,i_0} \sum_{i = 1}^{4} f_{i,i} 
= \sum_{i = 1}^{4}  f_{i_0,i_0} f_{i,i} = \sum_{i = 1}^{4}  f_{i_0,i} f_{i_0,i}
= \sum_{i = 1}^{4}  f_{i_0,i} f_{i_0,i}^*.
\end{align*}
we have $\chi(f_{i_0,i_0}) >0$. 
Put $a_i \coloneqq \frac{\chi(f_{i_0,i})}{\sqrt{\chi(f_{i_0,i_0})}}$. 
We have 
\begin{align*}
\sum_{i = 1}^{4} a_i^2 
&= \sum_{i = 1}^{4} \frac{\chi(f_{i_0,i})^2}{\chi(f_{i_0,i_0})} 
= \sum_{i = 1}^{4} \frac{\chi(f_{i_0,i_0}) \chi(f_{i,i})}{\chi(f_{i_0,i_0})} 
= \sum_{i = 1}^{4} \chi(f_{i,i}) = 1.
\end{align*}
We also have 
\begin{align*}
\chi(f_{i,j}) &= \frac{\chi(f_{i_0,i}) \chi(f_{i_0,j})}{\chi(f_{i_0,i_0})} 
= a_ia_j,
\end{align*}
for $i, j=1,2,3,4$. 
This shows that $\varphi$ is surjective. 

Since $\R P^3$ is compact and $X$ is Hausdorff, 
$\varphi \colon \R P^3 \to X$ is a homeomorphism. 
Thus we have shown that $A$ is isomorphic to $C(\R P^3)$. 
\end{proof}

Let $\{e_{i,j}\}_{i,j=1}^4$ be the matrix unit of $M_4(\C)$. 
Then $\{e_{i,j}\}_{i,j=1}^4$ satisfies the following relation 
denoted by $\mathcal{R}_{\text{e}}$; 
\begin{align*}
&e_{i,j}=e_{j,i}^*\quad \text{for all $i,j$,} \\
&e_{i,j}e_{k,l}=\delta_{j,k}e_{i,l} \quad \text{for all $i,j,k,l$,} \\
&\sum_{i=1}^4 e_{i,i}=1, 
\end{align*}
here $\delta_{j,k}$ is the Kronecker delta. 
It is well-known, and easy to see, that $M_4(\C)$ is the universal 
unital C*-algebra generated by $\{e_{i,j}\}_{i,j=1}^4$ 
satisfying $\mathcal{R}_{\text{e}}$. 

The \Ca $M_4(C(\R P^3)) = C(\R P^3,M_4(\C)) = C(\R P^3)\otimes M_4(\C)$ 
is the universal unital \Ca 
generated by $\{f_{i,j}\}_{i,j=1}^4$ and $\{e_{i,j}\}_{i,j=1}^4$ 
satisfying $\mathcal{R}_{\text{f}}$, $\mathcal{R}_{\text{e}}$ and 
the following relation 
denoted by $\mathcal{R}_{\text{fe}}$; 
\begin{align*}
&f_{i,j}e_{k,l}=e_{k,l}f_{i,j} \quad \text{for all $i,j,k,l$.} \\
\end{align*}

\section{Unitaries}\label{Sec:U}

\begin{definition}\label{Def:Uij}
For $i,j=1,2,3,4$, we define a unitary 
$U_{i,j}\in M_4(\C) \subset M_4(C(\R P^3))$ by 
\[
U_{i,j} \coloneqq \sum_{k=1}^4 \e(i,k)\e(k,j)e_{t_i(k),t_j(k)}
\]
\end{definition}

From a direct calculation, we have 
\begin{align*}
&U_{1,1}= 
\begin{pmatrix}
1 & 0 & 0 & 0  \\
0 & 1  & 0 & 0 \\
0 & 0 & 1 & 0  \\
0 & 0 & 0 & 1 
\end{pmatrix},&
&U_{1,2}=
\begin{pmatrix}
0 & 1 & 0 & 0 \\
-1 & 0 & 0 & 0 \\
0 & 0 & 0 & -1 \\
0 & 0 & 1 & 0 
\end{pmatrix},\\
&U_{1,3}=
\begin{pmatrix}
0 & 0 & 1 & 0 \\
0 & 0 & 0 & 1 \\
-1 & 0 & 0 & 0 \\
0 & -1 & 0 & 0 
\end{pmatrix},&
&U_{1,4}=
\begin{pmatrix}
0 & 0 & 0 & 1 \\
0 & 0 & -1 & 0 \\
0 & 1 & 0 & 0 \\
-1 & 0 & 0 & 0 
\end{pmatrix},\\
&U_{2,1}=
\begin{pmatrix}
0 & -1 & 0 & 0 \\
1 & 0 & 0 & 0 \\
0 & 0 & 0 & -1 \\
0 & 0 & 1 & 0 
\end{pmatrix},&
&U_{2,2}=
\begin{pmatrix}
1 & 0 & 0 & 0 \\
0 & 1 & 0 & 0 \\
0 & 0 & -1 & 0 \\
0 & 0 & 0 & -1 
\end{pmatrix},\\
&U_{2,3}=
\begin{pmatrix}
0 & 0 & 0 & -1 \\
0 & 0 & 1 & 0 \\
0 & 1 & 0 & 0 \\
-1 & 0 & 0 & 0 
\end{pmatrix},&
&U_{2,4}=
\begin{pmatrix}
0 & 0 & 1 & 0 \\
0 & 0 & 0 & 1 \\
1 & 0 & 0 & 0 \\
0 & 1 & 0 & 0 
\end{pmatrix},\\
%\end{align*}
%\begin{align*}
&U_{3,1}=
\begin{pmatrix}
0 & 0 & -1 & 0  \\
0 & 0 & 0 & 1 \\
1 & 0 & 0 & 0 \\
0 & -1 & 0 & 0
\end{pmatrix},&
&U_{3,2}=
\begin{pmatrix}
0 & 0 & 0 & 1 \\
0 & 0 & 1 & 0 \\
0 & 1 & 0 & 0 \\
1 & 0 & 0 & 0 
\end{pmatrix},\\
&U_{3,3}=
\begin{pmatrix}
1 & 0 & 0 & 0 \\
0 & -1 & 0 & 0 \\
0 & 0 & 1 & 0  \\
0 & 0 & 0 & -1
\end{pmatrix},&
&U_{3,4}=
\begin{pmatrix}
0 & -1 & 0 & 0 \\
-1 & 0 & 0 & 0 \\
0 & 0 & 0 & 1 \\
0  & 0 & 1 & 0	
\end{pmatrix},\\
& U_{4,1}=
\begin{pmatrix}
0 & 0 & 0 & -1 \\
0 & 0 & -1 & 0 \\
0 & 1 & 0 & 0 \\
1 & 0 & 0 & 0
\end{pmatrix},&
& U_{4,2}=
\begin{pmatrix}
0 & 0 & -1 & 0 \\
0 & 0 & 0 & 1 \\
-1 & 0 & 0 & 0 \\
0 & 1 & 0 & 0 
\end{pmatrix},\\
& U_{4,3}=
\begin{pmatrix}
0 & 1 & 0 & 0 \\
1 & 0 & 0  & 0 \\
0 & 0 & 0 & 1 \\
0 & 0 & 1 & 0 
\end{pmatrix},&
& U_{4,4}=
\begin{pmatrix}
1 & 0 & 0 & 0 \\
0 & -1 & 0 & 0 \\
0 & 0 & -1 & 0 \\
0 & 0 & 0 & 1
\end{pmatrix}.
\end{align*}

We have the following. 
We denote the transpose matrix of a matrix $M$ by $M^{\mathrm{T}}$. 

\begin{proposition}\label{Prop:Udef}
For $(a_1,a_2,a_3,a_4)\in \C^4$, 
\begin{equation*}
(b_1,b_2,b_3,b_4)^{\mathrm{T}} \coloneqq U_{i,j}(a_1,a_2,a_3,a_4)^{\mathrm{T}},
\end{equation*} 
satisfies 
$\sum_{k=1}^{4} b_k c_k  = c_i \left( \sum_{k=1}^{4}a_k c_k \right) c_j^*$.
\end{proposition}

\begin{proof}
For $i,j,k=1,2,3,4$, 
we have 
\begin{align*}
c_ic_{t_j(k)}&=\e(i,t_j(k))c_{t_i(t_j(k))}&
c_{t_i(k)}c_j&=\e(t_i(k),j)c_{t_j(t_i(k))}. 
\end{align*}
Hence $c_ic_{t_j(k)}c_j^*=\e(i,t_j(k))\e(t_i(k),j)^{-1}c_{t_i(k)}$. 
Since $\e(i,t_j(k))\e(k,j)=\e(i,k)\e(t_i(k),j)$, 
we have 
\begin{align*}
\e(i,t_j(k))\e(t_i(k),j)^{-1} = \e(i,k)\e(k,j)^{-1}=\e(i,k)\e(k,j)
\end{align*}
This shows that 
$U_{i,j} = \sum_{k=1}^4 \e(i,k)\e(k,j)e_{t_i(k),t_j(k)}$ 
satisfies the desired property. 
\end{proof}

\begin{proposition}\label{Prop:URu}
For $i,j,k,l = 1,2,3,4$, we have 
	\begin{equation*}
	U_{i,j} U_{k,l}=\varepsilon(i,k) \varepsilon(j ,l) U_{t_i(k), t_j(l)}.
	\end{equation*}
\end{proposition}

\begin{proof}
We have 
\begin{align*}
U_{i,j} U_{k,l}
&=\Big(\sum_{m=1}^4 \e(i,m)\e(m,j)e_{t_i(m),t_j(m)}\Big)
\Big(\sum_{n=1}^4 \e(k,n)\e(n,l)e_{t_k(n),t_l(n)}\Big)\\
&=\Big(\sum_{m=1}^4 \e(i,t_k(m))\e(t_k(m),j)e_{t_i(t_k(m)),t_j(t_k(m))}\Big)\\
&\phantom{=\Big(\sum_{m=1}^4 \e(i,t_k(m))\e(t_k(m),j)}
\Big(\sum_{n=1}^4 \e(k,t_j(n))\e(t_j(n),l)e_{t_k(t_j(n)),t_l(t_j(n))}\Big)\\
&=\sum_{m=1}^4 \e(i,t_k(m))\e(t_k(m),j)\e(k,t_j(m))\e(t_j(m),l)
e_{t_i(t_k(m)),t_l(t_j(m))}
\end{align*}
Since we have 
\begin{align*}
\e(i,t_k(m))\e(k,m)&=\e(i,k)\e(t_i(k),m),&
\e(k,t_j(m))\e(m,j)&=\e(k,m)\e(t_k(m),j)\\
\e(m,j)\e(t_j(m),l)&=\e(m,t_j(l))\e(j,l),
\end{align*}
we get 
\[
\e(i,t_k(m))\e(t_k(m),j)\e(k,t_j(m))\e(t_j(m),l)
=\e(i,k)\e(j,l)\e(t_i(k),m)\e(m,t_j(l)).
\]
Hence we obtain 
\begin{align*}
U_{i,j} U_{k,l}
&=\sum_{m=1}^4 \e(i,k)\e(j,l)\e(t_i(k),m)\e(m,t_j(l))
e_{t_i(t_k(m)),t_j(t_l(m))}\\
&=\e(i,k)\e(j,l)U_{t_i(k),t_j(l)}. 
\end{align*}
\end{proof}

One can prove this proposition using Proposition~\ref{Prop:Udef}. 

\section{Projections}\label{Sec:P}

\begin{definition}
We define $P_{1,1} \coloneqq 
\sum_{i,j=1}^4f_{i,j}e_{i,j} \in M_4(C(\R P^3))$. 
For $i,j =1,2,3,4$, 
we define $P_{i,j} \in M_4(C(\R P^3))$ by
\begin{equation*}
 P_{i,j} =U_{i,j} P_{1,1} U_{i,j}^*.
\end{equation*} 
\end{definition}
Note that $U_{1,1}=1$. 

\begin{proposition}\label{Prop:Pproj}
For each $i,j =1,2,3,4$, $P_{i,j}$ is a projection.
\end{proposition}

\begin{proof}
It suffices to show that $P_{1,1}$ is a projection. 
We have 
\begin{align*}
P_{1,1}^*= \sum_{i,j=1}^4f_{i,j}^*e_{i,j}^*
=\sum_{i,j=1}^4f_{j,i}e_{j,i}=P_{1,1},
\end{align*}
and 
\begin{align*}
P_{1,1}^2&=\sum_{i,j=1}^4f_{i,j}e_{i,j}\sum_{k,l=1}^4f_{k,l}e_{k,l}
=\sum_{i,j,k,l=1}^4 f_{i,j}e_{i,j}f_{k,l}e_{k,l}\\
&=\sum_{i,j,l=1}^4 f_{i,j}f_{j,l}e_{i,l}
=\sum_{i,j,l=1}^4 f_{i,l}f_{j,j}e_{i,l}
=\sum_{i,l=1}^4 f_{i,l}e_{i,l}=P_{1,1}.
\end{align*}
Hence $P_{1,1}$ is a projection. 
\end{proof}

\begin{proposition}\label{Prop:PU}
The set $\{P_{i,j}\}_{i,j=1}^4$ of projections 
and the set $\{U_{i,j}\}_{i,j=1}^4$ of unitaries 
satisfy $\mathcal{R}_{\text{up}}$. 
\end{proposition}

\begin{proof}
This follows from the computation 
\begin{align*}
U_{i,j}P_{k,l}U_{i,j}^*
&=U_{i,j}U_{k,l}P_{1,1}U_{k,l}^*U_{i,j}^*\\
&=(\e(i,k)\e(j,l))^2U_{t_i(k),t_j(l)}P_{1,1}U_{t_i(k),t_j(l)}^*
=P_{t_i(k),t_j(l)}
\end{align*}
using Proposition~\ref{Prop:URu}. 
\end{proof}

\begin{proposition}\label{Prop:Rp}
The set $\{P_{i,j}\}_{i,j=1}^4$ of projections 
satisfies $\mathcal{R}_{\text{p}}$. 
\end{proposition}

\begin{proof}
From Proposition~\ref{Prop:PU}, 
it suffices to show 
\begin{align*}
P_{1,1}+P_{1,2}+P_{1,3}+P_{1,4}&=1, &
P_{1,1}+P_{2,1}+P_{3,1}+P_{4,1}&=1. 
\end{align*}
This follows from the following direct computations 
\begin{align*}
&P_{1,1}= 
\begin{pmatrix}
f_{1,1} & f_{1,2} & f_{1,3} & f_{1,4} \\
f_{2,1} & f_{2,2} & f_{2,3} & f_{2,4} \\
f_{3,1} & f_{3,2} & f_{3,3} & f_{3,4} \\
f_{4,1} & f_{4,2} & f_{4,3} & f_{4,4}
\end{pmatrix},& &\\
&P_{1,2}=
\begin{pmatrix}
f_{2,2} & -f_{2,1} & -f_{2,4} & f_{2,3} \\
-f_{1,2} & f_{1,1} & f_{1,4} & -f_{1,3} \\
-f_{4,2} & f_{4,1} & f_{4,4} & -f_{4,3} \\
f_{3,2} & -f_{3,1} & -f_{3,4} & f_{3,3}
\end{pmatrix},&
&P_{2,1}=
\begin{pmatrix}
f_{2,2} & -f_{2,1} & f_{2,4} & -f_{2,3} \\
-f_{1,2} & f_{1,1} & -f_{1,4} & f_{1,3} \\
f_{4,2} & -f_{4,1} & f_{4,4} & -f_{4,3} \\
-f_{3,2} & f_{3,1} & -f_{3,4} & f_{3,3}
\end{pmatrix},\\
&P_{1,3}=
\begin{pmatrix}
f_{3,3} & f_{3,4} & -f_{3,1} & -f_{3,2} \\
f_{4,3} & f_{4,4} & -f_{4,1} & -f_{4,2} \\
-f_{1,3} & -f_{1,4} & f_{1,1} & f_{1,2} \\
-f_{2,3} & -f_{2,4} & f_{2,1} & f_{2,2}
\end{pmatrix},&
&P_{3,1}=
\begin{pmatrix}
f_{3,3} & -f_{3,4} & -f_{3,1} & f_{3,2} \\
-f_{4,3} & f_{4,4} & f_{4,1} & -f_{4,2} \\
-f_{1,3} & f_{1,4} & f_{1,1} & -f_{1,2} \\
f_{2,3} & -f_{2,4} & -f_{2,1} & f_{2,2}
\end{pmatrix},\\
&P_{1,4}=
\begin{pmatrix}
f_{4,4} & -f_{4,3} & f_{4,2} & -f_{4,1} \\
-f_{3,4} & f_{3,3} & -f_{3,2} & f_{3,1} \\
f_{2,4} & -f_{2,3} & f_{2,2} & -f_{2,1} \\
-f_{1,4} & f_{1,3} & -f_{1,2} & f_{1,1}
\end{pmatrix},&
&P_{4,1}=
\begin{pmatrix}
f_{4,4} & f_{4,3} & -f_{4,2} & -f_{4,1} \\
f_{3,4} & f_{3,3} & -f_{3,2} & -f_{3,1} \\
-f_{2,4} & -f_{2,3} & f_{2,2} & f_{2,1} \\
-f_{1,4} & -f_{1,3} & f_{1,2} & f_{1,1}
\end{pmatrix}.
\end{align*}
\end{proof}

By Proposition~\ref{Prop:URu}, Proposition~\ref{Prop:Pproj}, 
Proposition~\ref{Prop:PU} and Proposition~\ref{Prop:Rp}, 
we have a \shom 
$\varPhi \colon A(4) \rtimes_{\alpha}^{\text{tw}} (K \times K)
\to M_4(C(\R P^3))$ sending $p_{i,j}$ to $P_{i,j}$ and 
$u_{i,j}$ to $U_{i,j}$. 
In the next section, 
we construct the inverse map of $\varPhi$. 

\section{The inverse map}\label{Sec:Inv}

\begin{definition}
For $i,j=1,2,3,4$, 
we set 
\[
E_{i,j} \coloneqq \frac{1}{4}\sum_{k=1}^4 \e(i,k)\e(k,j)u_{t_i(k),t_j(k)} 
\in A(4) \rtimes_{\alpha}^{\text{tw}} (K \times K)
\]
\end{definition}

\begin{definition}
For $i,j=1,2,3,4$, 
we set 
\[
F_{i,j} \coloneqq \sum_{k=1}^4 E_{k,i}p_{1,1}E_{j,k}
\in A(4) \rtimes_{\alpha}^{\text{tw}} (K \times K). 
\]
\end{definition}

\begin{lemma}\label{Lem:uEu}
For $i,j=1,2,3,4$, 
we have $u_{i,1}E_{1,1}u_{1,j}=E_{i,j}$. 
For $i=1,2,3,4$, we have $u_{i,i}E_{1,1}=E_{1,1}u_{i,i}=E_{1,1}$. 
We also have $E_{1,1}^2=E_{1,1}$. 
\end{lemma}

\begin{proof}
We have $E_{1,1}=\frac{1}{4}\sum_{k=1}^4 u_{k,k}$. 
For $i,j=1,2,3,4$, 
we have 
\begin{align*}
u_{i,1}E_{1,1}u_{1,j}
&=\frac{1}{4}\sum_{k=1}^4 u_{i,1}u_{k,k}u_{1,j}
=\frac{1}{4}\sum_{k=1}^4 \e(i,k)\e(k,j)u_{t_i(k),t_j(k)}
=E_{i,j}.
\end{align*}
For $i=1,2,3,4$, we have
\begin{align*}
u_{i,i}E_{1,1}
&=\frac{1}{4}\sum_{k=1}^4 u_{i,i}u_{k,k}
=\frac{1}{4}\sum_{k=1}^4 \e(i,k)^2u_{t_i(k),t_i(k)}
=\frac{1}{4}\sum_{k=1}^4 u_{k,k}=E_{1,1}.
\end{align*}
Similarly, we get $E_{1,1}u_{i,i}=E_{1,1}$. 
Finally, 
we have $E_{1,1}^2=\frac{1}{4}\sum_{k=1}^4 u_{k,k}E_{1,1}=E_{1,1}$. 
\end{proof}

\begin{proposition}\label{Prop:Re}
The set $\{E_{i,j}\}_{i,j=1}^4$ satisfies $\mathcal{R}_{\text{e}}$. 
\end{proposition}

\begin{proof}
We have $E_{1,1}=\frac{1}{4}\sum_{k=1}^4 u_{k,k}$. 
We also have 
\begin{align*}
E_{2,2}&=\frac{1}{4}(u_{1,1}+u_{2,2}-u_{3,3}-u_{4,4})\\
E_{3,3}&=\frac{1}{4}(u_{1,1}-u_{2,2}+u_{3,3}-u_{4,4})\\
E_{4,4}&=\frac{1}{4}(u_{1,1}-u_{2,2}-u_{3,3}+u_{4,4}).
\end{align*}
Hence $\sum_{i=1}^4E_{i,i}=u_{1,1}=1$. 

It is easy to see $E_{1,1}^*=E_{1,1}$. 
For $i=1,2,3,4$, we have 
\[
E_{1,1}u_{i,1}^*=E_{1,1}u_{i,i}u_{i,1}^*=E_{1,1}u_{1,i}u_{i,1}u_{i,1}^*
=E_{1,1}u_{1,i}
\]
and $u_{1,i}^*E_{1,1}=u_{i,1}E_{1,1}$ similarly. 
Hence by Lemma~\ref{Lem:uEu}, we obtain 
\begin{align*}
E_{i,j}^*
&=(u_{i,1}E_{1,1}u_{1,j})^*
=u_{1,j}^*E_{1,1}u_{i,1}^*
=u_{j,1}E_{1,1}u_{1,i}=E_{j,i}
\end{align*}
for $i,j=1,2,3,4$. 

By Lemma~\ref{Lem:uEu}, we obtain 
\begin{align*}
E_{i,j}E_{j,k}
=u_{i,1}E_{1,1}u_{1,j}u_{j,1}E_{1,1}u_{1,k}
&=u_{i,1}E_{1,1}u_{j,j}E_{1,1}u_{1,k}\\
&=u_{i,1}E_{1,1}^2u_{1,k}
=u_{i,1}E_{1,1}u_{1,k}
=E_{i,k}
\end{align*}
for $i,j,k=1,2,3,4$. 
The proof ends if we show $E_{i,j}E_{k,l}=0$ 
for $i,j,k,l=1,2,3,4$ with $j \neq k$. 
It suffices to show $E_{1,1}u_{1,j}u_{k,1}E_{1,1}=0$ 
for $j,k=1,2,3,4$ with $j \neq k$. 
Since $u_{1,j}u_{k,1}=u_{k,j}=\e(k,t_k(j))u_{k,k}u_{1,t_k(j)}$, 
it suffices to show $E_{1,1}u_{1,j}E_{1,1}=0$ 
for $j=2,3,4$. 
For $j=2$, we get 
\begin{align*}
4E_{1,1}u_{1,2}E_{1,1}
&=\sum_{k=1}^4 u_{k,k}u_{1,2}E_{1,1}\\
&=u_{1,2}E_{1,1}+u_{1,2}u_{2,2}E_{1,1}-u_{1,2}u_{3,3}E_{1,1}-u_{1,2}u_{4,4}E_{1,1}
\\
&=0
\end{align*}
%\begin{align*}
%16E_{1,1}u_{1,2}E_{1,1}
%&=\sum_{k,m=1}^4 u_{k,k}u_{1,2}u_{m,m}\\
%&=u_{1,2}-u_{2,1}+u_{3,4}-u_{4,3}-u_{2,1}+u_{1,2}-u_{4,3}+u_{3,4}\\
%&\phantom{=}-u_{3,4}+u_{4,3}-u_{1,2}+u_{2,1}+u_{4,3}-u_{3,4}+u_{2,1}-u_{1,2}\\
%&=0
%\end{align*}
By similar computations, 
we get $E_{1,1}u_{1,3}E_{1,1}=E_{1,1}u_{1,4}E_{1,1}=0$. 
This completes the proof. 
\end{proof}

\begin{proposition}\label{Prop:Rf}
The set $\{F_{i,j}\}_{i,j=1}^4$ 
satisfy $\mathcal{R}_{\text{f}}$. 
\end{proposition}

\begin{proof}
For $i,j=1,2,3,4$, 
we have 
\begin{align*}
F_{i,j}^*
=\Big(\sum_{k=1}^4 E_{k,i}p_{1,1}E_{j,k}\Big)^*
&=\sum_{k=1}^4 E_{j,k}^*p_{1,1}^*E_{k,i}^*\\
&=\sum_{k=1}^4 E_{k,j}p_{1,1}E_{i,k}
=F_{j,i}.
\end{align*}
Next, we show $F_{i,j} =F_{j,i}$ for $i,j=1,2,3,4$. 
We are going to prove $F_{2,4} =F_{4,2}$. 
The other 5 cases can be proved similarly. 
To show that $F_{2,4} =F_{4,2}$, 
it suffices to show $E_{1,2}p_{1,1}E_{4,1}=E_{1,4}p_{1,1}E_{2,1}$ 
because it implies $E_{k,2}p_{1,1}E_{4,k}=E_{k,4}p_{1,1}E_{2,k}$ 
for $k=1,2,3,4$ by multiplying $E_{k,1}$ from left and $E_{1,k}$ from right. 
By Lemma~\ref{Lem:uEu}, we have 
\begin{align*}
4E_{1,2}p_{1,1}E_{4,1}
&=(u_{1,2}-u_{2,1}-u_{3,4}+u_{4,3})p_{1,1}u_{4,1}E_{1,1}\\
&=(p_{1,2}u_{1,2}-p_{2,1}u_{2,1}-p_{3,4}u_{3,4}+p_{4,3}u_{4,3})u_{4,1}E_{1,1}\\
&=(p_{1,2}u_{4,2}+p_{2,1}u_{3,1}-p_{3,4}u_{2,4}-p_{4,3}u_{1,3})E_{1,1}\\
&=(p_{1,2}u_{1,3}u_{4,4}-p_{2,1}u_{1,3}u_{3,3}+p_{3,4}u_{1,3}u_{2,2}-p_{4,3}u_{1,3})E_{1,1}\\
&=(p_{1,2}-p_{2,1}+p_{3,4}-p_{4,3})u_{1,3}E_{1,1}\\
4E_{1,4}p_{1,1}E_{2,1}
&= (u_{1,4}-u_{2,3}+u_{3,2}-u_{4,1})p_{1,1}u_{2,1}E_{1,1}\\
&=(p_{1,4}u_{1,4}-p_{2,3}u_{2,3}+p_{3,2}u_{3,2}-p_{4,1}u_{4,1})u_{2,1}E_{1,1}\\
&=(p_{1,4}u_{2,4}+p_{2,3}u_{1,3}-p_{3,2}u_{4,2}-p_{4,1}u_{3,1})E_{1,1}\\
&=(-p_{1,4}u_{1,3}u_{2,2}+p_{2,3}u_{1,3}-p_{3,2}u_{1,3}u_{4,4}+p_{4,1}u_{1,3}u_{3,3})E_{1,1}\\
&=(-p_{1,4}+p_{2,3}-p_{3,2}+p_{4,1})u_{1,3}E_{1,1}.
\end{align*}
Since 
\begin{align*}
p_{1,1}&+p_{1,2}+p_{1,3}+p_{1,4}+p_{3,1}+p_{3,2}+p_{3,3}+p_{3,4}\\
&=2=p_{1,1}+p_{2,1}+p_{3,1}+p_{4,1}+p_{1,3}+p_{2,3}+p_{3,3}+p_{4,3}, 
\end{align*}
we have 
\[
p_{1,2}-p_{2,1}+p_{3,4}-p_{4,3}
=-p_{1,4}+p_{2,3}-p_{3,2}+p_{4,1}. 
\]
Therefore, we obtain $E_{1,2}p_{1,1}E_{4,1} = E_{1,4}p_{1,1}E_{2,1}$. 
Thus we have proved $F_{2,4} =F_{4,2}$. 

Next we show $F_{i,j}F_{k,l}=F_{i,k}F_{j,l}$ for $i,j,k,l=1,2,3,4$, 
To show this, 
it suffices to show $p_{1,1}E_{j,k}p_{1,1}=p_{1,1}E_{k,j}p_{1,1}$ 
for $j,k=1,2,3,4$. 
We are going to prove $p_{1,1}E_{3,4}p_{1,1}=p_{1,1}E_{4,3}p_{1,1}$. 
The other 5 cases can be proved similarly. 
This follows from the following computation 
\begin{align*}
4p_{1,1}E_{3,4}p_{1,1}
&=p_{1,1}(u_{3,4}+u_{4,3}-u_{1,2}-u_{2,1})p_{1,1}\\
&=p_{1,1}(u_{3,4}+u_{4,3})p_{1,1}-p_{1,1}p_{1,2}u_{1,2}-p_{1,1}p_{2,1}u_{2,1}\\
&=p_{1,1}(u_{3,4}+u_{4,3})p_{1,1},\\
4p_{1,1}E_{4,3}p_{1,1}
&=p_{1,1}(u_{4,3}+u_{3,4}+u_{2,1}+u_{1,2})p_{1,1}\\
&=p_{1,1}(u_{3,4}+u_{4,3})p_{1,1}+p_{1,1}p_{2,1}u_{2,1}+p_{1,1}p_{1,2}u_{1,2}\\
&=p_{1,1}(u_{3,4}+u_{4,3})p_{1,1}.
\end{align*}

Finally we show $\sum_{i=1}^4 F_{i,i}=1$. 
For $i=1,2,3,4$, we have 
\begin{align*}
F_{i,i}
&=\sum_{k=1}^4 E_{k,i}p_{1,1}E_{i,k}
=\sum_{k=1}^4 u_{k,1}E_{1,1}u_{1,i}p_{1,1}u_{i,1}E_{1,1}u_{1,k}\\
&=\sum_{k=1}^4 u_{k,1}E_{1,1}p_{1,i}u_{1,i}u_{i,1}E_{1,1}u_{1,k}
=\sum_{k=1}^4 u_{k,1}E_{1,1}p_{1,i}u_{i,i}E_{1,1}u_{1,k}\\
&=\sum_{k=1}^4 u_{k,1}E_{1,1}p_{1,i}E_{1,1}u_{1,k}.
\end{align*}
Hence we obtain 
\begin{align*}
\sum_{i=1}^4 F_{i,i}
&= \sum_{i=1}^4 \sum_{k=1}^4 u_{k,1}E_{1,1}p_{1,i}E_{1,1}u_{1,k}\\
&= \sum_{k=1}^4 u_{k,1}E_{1,1}^2u_{1,k}
= \sum_{k=1}^4 u_{k,1}E_{1,1}u_{1,k}
= \sum_{k=1}^4 E_{k,k}=1 
\end{align*}
by Lemma~\ref{Lem:uEu} and Proposition~\ref{Prop:Re}. 
We are done. 
\end{proof}

\begin{proposition}\label{Prop:Rfe}
The sets $\{E_{i,j}\}_{i,j=1}^4$ and $\{F_{i,j}\}_{i,j=1}^4$ 
satisfy $\mathcal{R}_{\text{fe}}$. 
\end{proposition}

\begin{proof}
For $i,j,k,l=1,2,3,4$, 
we have $E_{i,j}F_{k,l}=F_{k,l}E_{i,j}$ because 
\begin{align*}
E_{i,j}F_{k,l}&=E_{i,j}\sum_{m=1}^4 E_{m,k}p_{1,1}E_{l,m}
=E_{i,k}p_{1,1}E_{l,j},\\
F_{k,l}E_{i,j}&=\sum_{m=1}^4 E_{m,k}p_{1,1}E_{l,m}E_{i,j}
=E_{i,k}p_{1,1}E_{l,j}
\end{align*}
by Proposition~\ref{Prop:Re}. 
\end{proof}

By Proposition~\ref{Prop:Re}, Proposition~\ref{Prop:Rf}
and Proposition~\ref{Prop:Rfe}, 
we have a \shom $\varPsi\colon M_4(C(\R P^3))\to 
A(4) \rtimes_{\alpha}^{\text{tw}} (K \times K)$ 
sending $f_{i,j}$ to $F_{i,j}$ and 
$e_{i,j}$ to $E_{i,j}$. 

We are going to see that this map $\varPsi$ is the inverse of $\varPhi$. 
We first show $\varPsi\circ \varPhi = \id_{A(4) \rtimes_{\alpha}^{\text{tw}} (K \times K)}$. 

\begin{proposition}\label{Prop:PsiPhi}
For $x \in A(4) \rtimes_{\alpha}^{\text{tw}} (K \times K)$, 
we have $\varPsi( \varPhi(x)) = x$. 
\end{proposition}

\begin{proof}
For $i,j=1,2,3,4$, 
we have 
\begin{align*}
\varPsi( \varPhi(u_{i,j}))
&=\varPsi(U_{i,j})
=\sum_{k=1}^4 \e(i,k)\e(k,j)\varPsi(e_{t_i(k),t_j(k)})\\
&=\sum_{k=1}^4 \e(i,k)\e(k,j)E_{t_i(k),t_j(k)}\\
&=\frac{1}{4}\sum_{k=1}^4 \e(i,k)\e(k,j)
\sum_{m=1}^4\e(t_i(k),m)\e(m,t_j(k))u_{t_i(t_k(m)),t_j(t_k(m))}\\
&=\frac{1}{4}\sum_{k=1}^4\sum_{l=1}^4 \e(i,k)\e(k,j)
\e(t_i(k),t_k(l))\e(t_k(l),t_j(k))u_{t_i(l),t_j(l)}.
\end{align*}
Since we have 
\begin{align*}
\frac{1}{4}\sum_{k=1}^4
&\e(i,k)\e(k,j)
\e(t_i(k),t_k(l))\e(t_k(l),t_j(k))\\
&=\frac{1}{4}\sum_{k=1}^4
\e(i,k)
\e(t_i(k),t_k(l))\e(t_k(l),t_j(k))\e(k,j)\\
&=\frac{1}{4}\sum_{k=1}^4
\e(i,l)
\e(k,t_k(l))\e(t_k(l),k)\e(l,j)
=\delta_{l,1},
\end{align*}
we obtain $\varPsi(\varPhi(u_{i,j}))=u_{i,j}$. 
By the computation in the proof of Proposition~\ref{Prop:Rfe}, 
we have 
\begin{align*}
\varPsi(P_{1,1})
=\varPsi\Big(\sum_{i,j=1}^4f_{i,j}e_{i,j}\Big)
=\sum_{i,j=1}^4F_{i,j}E_{i,j}
=\sum_{i,j=1}^4 E_{i,i}p_{1,1}E_{j,j}
=p_{1,1}. 
\end{align*}
For $i,j=1,2,3,4$, 
we have 
\begin{align*}
\varPsi( \varPhi(p_{i,j}))
&=\varPsi(P_{i,j})
=\varPsi(U_{i,j})\varPsi(P_{1,1})\varPsi(U_{i,j})^*
=u_{i,j}p_{1,1}u_{i,j}^*=p_{i,j}.
\end{align*}
These show that $\varPsi( \varPhi(x)) = x$
for all $x \in A(4) \rtimes_{\alpha}^{\text{tw}} (K \times K)$. 
\end{proof}

\begin{proposition}\label{Prop:PhiPsi}
For $x \in M_4(C(\R P^3))$, 
we have $\varPhi( \varPsi(x)) = x$. 
\end{proposition}

\begin{proof}
For $i,j=1,2,3,4$, 
we have 
\begin{align*}
\varPhi( \varPsi(e_{i,j}))
&=\varPhi(E_{i,j})
=\frac{1}{4}\sum_{k=1}^4 \e(i,k)\e(k,j)\varPhi(u_{t_i(k),t_j(k)})\\
&=\frac{1}{4}\sum_{k=1}^4 \e(i,k)\e(k,j)U_{t_i(k),t_j(k)}\\
&=\frac{1}{4}\sum_{k=1}^4 \e(i,k)\e(k,j)
\sum_{m=1}^4 \e(t_i(k),m)\e(m,t_j(k))e_{t_i(t_k(m)),t_j(t_k(m))}\\
&=\frac{1}{4}\sum_{k=1}^4 \sum_{l=1}^4 \e(i,k)\e(k,j)
\e(t_i(k),t_k(l))\e(t_k(l),t_j(k))e_{t_i(l),t_j(l)}\\
&=e_{i,j}
\end{align*}
as in the proof of Proposition~\ref{Prop:PsiPhi}. 
For $i,j=1,2,3,4$, 
we have 
\begin{align*}
\varPhi( \varPsi(f_{i,j}))
&=\varPhi(F_{i,j})
=\sum_{k=1}^4 \varPhi(E_{k,i})\varPhi(p_{1,1})\varPhi(E_{j,k})\\
&=\sum_{k=1}^4 e_{k,i}P_{1,1}e_{j,k}\\
&=\sum_{k=1}^4 e_{k,i}\Big(\sum_{l,m=1}^4f_{l,m}e_{l,m}\Big)e_{j,k}\\
&=\sum_{k=1}^4f_{i,j}e_{k,k}=f_{i,j}. 
\end{align*}
These show that $\varPhi( \varPsi(x)) = x$
for all $x \in M_4(C(\R P^3))$. 
\end{proof}

By these two propositions, 
we get Theorem~\ref{MainThm1}. 
As its corollary, we have the following. 

\begin{corollary}[{cf.\ \cite[Theorem~4.1]{BC}}]\label{Cor:faithful}
There is an injective \shom $A(4) \to M_4(C(\R P^3))$. 
\end{corollary}

\begin{proof}
This follows from Theorem~\ref{MainThm1} 
because the \shom $A(4) \to A(4) \rtimes_{\alpha}^{\text{tw}} (K \times K)$ 
is injective. 
\end{proof}

One can see that the injective \shom constructed in this corollary 
is nothing but the Pauli representation constructed in \cite{BM} 
and considered in \cite{BC}. 
Note that Banica and Collins remarked after \cite[Definition~2.1]{BC} 
that the target of the Pauli representation is replaced by $M_4(C(SO_3))$
instead of $M_4(C(SU_2))$. 
Here $SO_3$ is a homeomorphic to $\R P^3$ whereas 
$SU_2$ is a homeomorphic to $S^3$.

\section{Action}\label{Sec:Act}

One can see that the dual group of $K\times K$ is isomorphic to $K \times K$ 
using the product of the cocycle $\e$ (see below). 

\begin{table}[htb]
\begin{center}
\begin{tabular}{|c||c|c|c|c|} \hline
$i\diagdown j$ & $1$ &$2$  & $3$ & $4$ \\ \hline \hline
$1$& 1  &  1 & 1  & 1  \\ \hline
$2$& 1  &  1 &  $-1$ & $-1$ \\ \hline
$3$& 1  & $ -1$ & 1 & $-1$ \\ \hline
$4$& \,1\, &  $-1$ &  $-1$& 1 \\ \hline
\end{tabular}
\end{center}
\caption{Values of $\varepsilon(i,j)\varepsilon(j,i)$}
\label{table:cocycle2}
\end{table} 

Let $\ha\colon K\times K \curvearrowright 
A(4)\rtimes_{\alpha}^{\text{tw}} (K \times K)$ be 
the dual action of $\alpha$. 
Namely $\ha$ is determined by the following equation 
for all $i,j,k,l$
\begin{align*}
\ha_{i,j}(p_{k,l})&=p_{k,l},&
\ha_{i,j}(u_{k,l})&=\e(i,k)\e(k,i)\e(j,l)\e(l,j)u_{k,l},&
\end{align*}
where we write $\ha_{(t_i,t_j)}$ as $\ha_{i,j}$. 

For $i,j=1,2,3,4$, 
define $\sigma_{i,j}\colon \R P^3 \to \R P^3$ by 
$\sigma_{i,j}([a_1,a_2,a_3,a_4])=[b_1,b_2,b_3,b_4]$ 
for $[a_1,a_2,a_3,a_4] \in \R P^3$ where $(b_1,b_2,b_3,b_4)\in S^3$ is 
determined by 
\begin{equation*}
(b_1,b_2,b_3,b_4)^{\mathrm{T}} = U_{i,j}(a_1,a_2,a_3,a_4)^{\mathrm{T}},
\end{equation*} 
in other words
$\sum_{k=1}^{4} b_k c_k = c_i \left( \sum_{k=1}^{4}a_k c_k \right) c_j^*$ 
by Proposition~\ref{Prop:Udef}.
Let $\beta\colon K\times K \curvearrowright M_4(C(\R P^3))$
be the action determined by 
$\beta_{i,j}(F)=\Ad U_{i,j}\circ F\circ \sigma_{i,j}$ 
for $F \in M_4(C(\R P^3))=C(\R P^3,M_4(\C))$ 
where we write $\beta_{(t_i,t_j)}$ as $\beta_{i,j}$. 

\begin{proposition}\label{Prop:equivariant}
The \shom $\varPhi\colon A(4) \rtimes_{\alpha}^{\text{tw}} (K \times K)
\to M_4(C(\R P^3))$ is equivariant with respect to $\ha$ and $\beta$. 
\end{proposition}

\begin{proof}
For $i,j=1,2,3,4$, 
we have 
$P_{1,1}\circ \sigma_{i,j}=\Ad U_{i,j} \circ P_{1,1}$. 
In fact for $[a_1,a_2,a_3,a_4] \in \R P^3$, 
on one hand we have 
\begin{align*}
(P_{1,1}\circ \sigma_{i,j})([a_1,a_2,a_3,a_4])
&=(b_1,b_2,b_3,b_4)^{\mathrm{T}}(b_1,b_2,b_3,b_4),\\
\end{align*}
where 
\begin{equation*}
(b_1,b_2,b_3,b_4)^{\mathrm{T}} = U_{i,j}(a_1,a_2,a_3,a_4)^{\mathrm{T}},
\end{equation*} 
and on the other hand we have
\begin{align*}
(\Ad U_{i,j} \circ P_{1,1})([a_1,a_2,a_3,a_4])
&=U_{i,j} (a_1,a_2,a_3,a_4)^{\mathrm{T}}(a_1,a_2,a_3,a_4)U_{i,j}^*
\end{align*}
here note $U_{i,j}^*=U_{i,j}^{\mathrm{T}}$ 
because the entries of $U_{i,j}$ are $-1$, $0$ or $1$. 
For $i,j,k,l=1,2,3,4$, 
we have 
\begin{align*}
\beta_{i,j}(P_{k,l})
&=\Ad U_{i,j}\circ (\Ad U_{k,l} \circ P_{1,1})\circ \sigma_{i,j}\\
&=\Ad U_{i,j}\circ \Ad U_{k,l} \circ \Ad U_{i,j} \circ P_{1,1}\\
&= \Ad (U_{i,j}U_{k,l}U_{i,j}) \circ P_{1,1}\\
&= \Ad U_{k,l} \circ P_{1,1}=P_{k,l}. 
\end{align*}
For $i,j,k,l=1,2,3,4$, 
we also have 
\begin{align*}
\beta_{i,j}(U_{k,l})
&=\Ad U_{i,j}\circ U_{k,l} \circ \sigma_{i,j}\\
&=U_{i,j}U_{k,l}U_{i,j}^* \\
&=\e(i,k)\e(j,l)U_{t_i(k),t_j(l)}U_{i,j}^* \\
&=\e(i,k)\e(j,l)\e(k,i)^{-1}\e(l,j)^{-1}U_{k,l}U_{i,j}U_{i,j}^*\\
&=\e(i,k)\e(j,l)\e(k,i)\e(l,j)U_{k,l}
\end{align*}
here note that $U_{k,l} \in M_4(C(\R P^3))$ is a constant function. 
These complete the proof. 
\end{proof}

The following is the second main theorem. 

\begin{theorem}\label{MainThm2}
The fixed point algebra $M_4(C(\R P^3))^{\beta}$ of the action $\beta$ is 
isomorphic to $A(4)$. 
\end{theorem}

\begin{proof}
This follows from Theorem~\ref{MainThm1} 
and Proposition~\ref{Prop:equivariant} 
because the fixed point algebra 
$\big(A(4) \rtimes_{\alpha}^{\text{tw}} (K \times K)\big)^{\ha}$ 
of $\ha$ is $A(4)$. 
\end{proof}

\section{Quotient Space $\R P^3 / (K \times K)$}
\label{Sec:quotient}

\begin{definition}
We set $A \coloneqq M_4(C(\R P^3))^{\beta}$. 
\end{definition}

By Theorem~\ref{MainThm2}, 
the \Ca $A(4)$ is isomorphic to $A$. 
From this section, 
we compute the structure of $A$ 
and its K-groups. 

In this section, we study the quotient Space $\R P^3 / (K \times K)$ 
of $\R P^3$ by the action $\sigma$ of $K \times K$. 
In \cite{osugi}, 
it is proved that this quotient space $\R P^3 / (K \times K)$ is homeomorphic to $S^3$. 

\begin{definition}
We denote by $X$ the quotient space $\R P^3 / (K \times K)$ of 
the action $\sigma$ of $K \times K$.
We denote by $\pi\colon \R P^3 \to X$ the quotient map.
\end{definition}

We use the following lemma later. 

\begin{lemma}\label{Lem:PaPa}
For $i,j=2,3,4$ and $[a_1,a_2,a_3,a_4] \in \R P^3$ with $\sigma_{i,j}([a_1,a_2,a_3,a_4])=[a_1,a_2,a_3,a_4]$, 
we have $P_{k,l}([a_1,a_2,a_3,a_4])=P_{t_i(k),t_j(l)}([a_1,a_2,a_3,a_4])$ 
for $k,l=1,2,3,4$.
\end{lemma}

\begin{proof}
This follows from 
\begin{align*}
P_{k,l}([a_1,a_2,a_3,a_4])
&=\beta_{i,j}(P_{k,l})([a_1,a_2,a_3,a_4])\\
&=\Ad U_{i,j} \big(P_{k,l} (\sigma_{i,j}([a_1,a_2,a_3,a_4]))\big)\\
&=\Ad U_{i,j} \big(P_{k,l} ([a_1,a_2,a_3,a_4])\big),\\
P_{t_i(k),t_j(l)}([a_1,a_2,a_3,a_4])
&=\big(\Ad U_{i,j}(P_{k,l})\big) ([a_1,a_2,a_3,a_4])\\
&=\Ad U_{i,j} \big(P_{k,l} ([a_1,a_2,a_3,a_4])\big).\qedhere
\end{align*}
\end{proof}

\begin{definition}
For each $i,j =2,3,4$, 
define 
\[
\tF_{i,j} \coloneqq \{[a_1,a_2,a_3,a_4] \in \R P^3\mid 
\sigma_{i,j}([a_1,a_2,a_3,a_4])=[a_1,a_2,a_3,a_4]\} \subset \R P^3
\]
to be the set of fixed points of $\sigma_{i,j}$, 
and define $F_{i,j} \subset X $ to be the image $\pi (\tilde{F}_{i,j})$.
\end{definition}

We have $\tilde{F}_{i,j} = \pi^{-1}(F_{i,j})$. 
The following two propositions can be proved by direct computation 
using the computation of $U_{i,j}$ after Definition~\ref{Def:Uij}

\begin{proposition}\label{Prop:}
For each $i=2,3,4$, 
$\sigma_{1,i}$ and $\sigma_{i,1}$ have no fixed points. 
\end{proposition}

\begin{proposition}\label{Prop:2circle}
For each $i,j =2,3,4$, 
$\tF_{i,j}$ is homeomorphic to a disjoint union of two circles. 
More precisely, we have
\begin{align*}
&\tF_{2,2}= \{[a,b,0,0],[0,0,a,b] \in \R P^3 \mid a,b \in \R,\ a^2 +b^2=1 \} \\
&\tF_{2,3}= \{ [a,b,-b,a],[a,b,b,-a] \in \R P^3 \mid a,b \in \R,\ 2(a^2 +b^2)=1 \} \\
&\tF_{2,4}= \{[a,b,a,b],[a,b,-a,-b] \in \R P^3 \mid a,b \in \R,\ 2(a^2 +b^2)=1 \} \\
&\tF_{3,2}= \{[a,b,b,a], [a,b,-b,-a] \in \R P^3  \mid a,b \in \R,\ 2(a^2 +b^2)=1 \} \\
&\tF_{3,3}= \{[a,0,b,0],[0,a,0,b] \in \R P^3 \mid a,b \in \R,\ a^2 +b^2=1 \} \\
&\tF_{3,4}= \{[a,a,b,-b],[a,-a,b,b] \in \R P^3 \mid a,b \in \R,\ 2(a^2 +b^2)=1 \} \\
&\tF_{4,2}= \{[a,b,a,-b],[a,b,-a,b] \in \R P^3 \mid a,b \in \R,\ 2(a^2 +b^2)=1 \} \\
&\tF_{4,3}= \{[a,a,b,b], [a,-a,b, -b] \in \R P^3 \mid a,b \in \R,\ 2(a^2 +b^2)=1 \} \\
&\tF_{4,4}= \{ [a,0,0,b],[0,a,b,0] \in \R P^3 \mid a,b \in \R,\ a^2 +b^2=1 \} 
\end{align*}
\end{proposition}

\begin{definition}
We set $\tF \coloneqq \bigcup_{i,j =2}^4 \tF_{i,j}$ and 
$F \coloneqq \bigcup_{i,j =2}^4 F_{i,j}$. 
We also set $\tO \coloneqq \R P^3 \setminus \tF$ and 
$O \coloneqq X \setminus F$. 
\end{definition}

We have $\tF=\pi^{-1}(F)$ and hence $\tO=\pi^{-1}(O)$. 
Note that $\tO$ is the set of points $[a_1,a_2,a_3,a_4] \in \R P^3$ 
such that $\sigma_{i,j}([a_1,a_2,a_3,a_4]) \neq [a_1,a_2,a_3,a_4]$ 
for all $i,j=1,2,3,4$ other than $(i,j)=(1,1)$. 
Note also that $\tF$ and $F$ are closed, and hence $\tO$ and $O$ are open. 

\begin{definition}
For each $i_2,i_3,i_4$ with $\{i_2,i_3,i_4  \}= \{2,3,4\}$,
define $\tF_{(i_2 i_3 i_4)} \subset \R P^3$ by
\begin{equation*}
\tF_{(i_2 i_3 i_4)} \coloneqq \tF_{i_2,2} \cap \tF_{i_3, 3} \cap \tF_{i_4,4},
\end{equation*}
and define $F_{(i_2 i_3 i_4)} \subset X$ to be the image $\pi (\tF_{(i_2 i_3 i_4)})$.
\end{definition}

\begin{proposition}\label{Prop:4points}
For each $i_2,i_3,i_4$ with $\{i_2,i_3,i_4  \}= \{2,3,4\}$,
we have 
\begin{equation*}
\tF_{(i_2 i_3 i_4)}= \tF_{i_2,2} \cap \tF_{i_3,3}= \tF_{i_2,2} \cap \tF_{i_4,4}= \tF_{i_3,3} \cap \tF_{i_4,4}. 
\end{equation*}
We also have 
\begin{align*}
&\tF_{(234)} = \big \{[1,0,0,0]  ,[0,1,0,0],[0,0,1,0],[0,0,0,1] \big\}, \\
&\tF_{(342)} = \left \{\Big[\frac{1}{2},\frac{1}{2},\frac{1}{2},\frac{1}{2}\Big], \Big[\frac{1}{2},\frac{1}{2},-\frac{1}{2},-\frac{1}{2}\Big],\Big[\frac{1}{2},-\frac{1}{2},-\frac{1}{2},\frac{1}{2}\Big], \Big[\frac{1}{2},-\frac{1}{2},\frac{1}{2},-\frac{1}{2}\Big] \right \}, \\
&\tF_{(423)} = \left \{\Big[-\frac{1}{2},\frac{1}{2},\frac{1}{2},\frac{1}{2}\Big], \Big[\frac{1}{2},-\frac{1}{2},\frac{1}{2},\frac{1}{2}\Big],\Big[\frac{1}{2},\frac{1}{2},-\frac{1}{2},\frac{1}{2}\Big], \Big[\frac{1}{2},\frac{1}{2},\frac{1}{2},-\frac{1}{2}\Big] \right \}, \\
&\tF_{(243)}= \left \{\Big[\frac{1}{\sqrt{2}},\frac{1}{\sqrt{2}},0,0\Big],\Big[\frac{1}{\sqrt{2}},-\frac{1}{\sqrt{2}},0,0\Big], \Big[0,0,\frac{1}{\sqrt{2}},\frac{1}{\sqrt{2}}\Big], \Big[0,0,\frac{1}{\sqrt{2}},-\frac{1}{\sqrt{2}}\Big] \right \}, \\
&\tF_{(432)}  = \left \{\Big[\frac{1}{\sqrt{2}},0,\frac{1}{\sqrt{2}},0\Big], \Big[\frac{1}{\sqrt{2}},0,-\frac{1}{\sqrt{2}},0\Big], \Big[0,\frac{1}{\sqrt{2}},0,\frac{1}{\sqrt{2}}\Big], \Big[0,\frac{1}{\sqrt{2}},0,-\frac{1}{\sqrt{2}}\Big] \right \}, \\
&\tF_{(324)} = \left \{\Big[\frac{1}{\sqrt{2}},0,0,\frac{1}{\sqrt{2}}\Big], \Big[\frac{1}{\sqrt{2}},0,0,-\frac{1}{\sqrt{2}}\Big], \Big[0,\frac{1}{\sqrt{2}},\frac{1}{\sqrt{2}},0\Big], \Big[0,\frac{1}{\sqrt{2}},-\frac{1}{\sqrt{2}},0\Big] \right \}.
\end{align*}
\end{proposition}

\begin{proof}
This follows from Proposition~\ref{Prop:2circle}. 
\end{proof}

\begin{proposition}
For each $i_2,i_3,i_4$ with $\{i_2,i_3,i_4  \}= \{2,3,4\}$,
$F_{(i_2 i_3 i_4)}$ consists of one point. 
\end{proposition}

\begin{proof}
This follows from Proposition~\ref{Prop:4points}. 
\end{proof}

\begin{definition}
For each $i_2,i_3,i_4$ with $\{i_2,i_3,i_4  \}= \{2,3,4\}$,
we set $x_{(i_2 i_3 i_4)} \in X$ 
by $F_{(i_2 i_3 i_4)}=\{x_{(i_2 i_3 i_4)}\}$. 
\end{definition}

\begin{proposition}
For each $i,j=2,3,4$, 
$F_{i,j}$ is homeomorphic to a closed interval 
whose endpoints are $x_{(i_2 i_3 i_4)}$ with $i_j=i$, 
\end{proposition}

\begin{proof}
This follows from Proposition~\ref{Prop:2circle}. 
See also Figure~\ref{fig:EF} and the remark around it. 
\end{proof}

Note that $F \subset X$ is the complete bipartite graph 
between $\{x_{(234)},x_{(342)},x_{(423)}\}$ 
and $\{x_{(243)},x_{(432)},x_{(324)},\}$. 
See Figure~\ref{fig:EF}. 

\begin{definition}
For $i,j=2,3,4$, 
we define 
\begin{equation*}
F_{i,j}^\circ \coloneqq F_{i,j} \setminus \{x_{(i_2i_3i_4)} \mid i_j =i  \},
\end{equation*}
and define 
\begin{equation*}
F^{\circ} \coloneqq \bigcup_{i_j =2}^4 F_{i,j}^\circ, \qquad 
F^{\text{\textbullet}} \coloneqq 
\{x_{(234)},x_{(243)},x_{(324)},x_{(342)},x_{(423)},x_{(432)} \}.
\end{equation*}
\end{definition}

\begin{definition}
We set $\tF_{i,j}^{\circ} \coloneqq \pi^{-1}(F_{i,j}^{\circ})$ 
for $i,j=2,3,4$, 
$\tF^{\circ} \coloneqq \pi^{-1}(F^{\circ})$ and 
$\tF^{\text{\textbullet}} \coloneqq \pi^{-1}(F^{\text{\textbullet}})$. 
\end{definition}

\section{Exact sequences}\label{Sec:Exact}

For a locally compact subset $Y$ of $\R P^3$ 
which is invariant under the action $\sigma$, 
the action $\beta\colon K\times K \curvearrowright M_4(C(\R P^3))$ 
induces the action $K\times K \curvearrowright M_4(C_0(Y))$ 
which is also denoted by $\beta$. 
We use the following lemma many times. 

\begin{lemma}\label{Lem:exseq}
Let $Y$ be a locally compact subset of $\R P^3$ 
which is invariant under the action $\sigma$. 
Let $Z$ be a closed subset of $Y$ 
which is invariant under the action $\sigma$. 
Then we have a a short exact sequence 
\begin{equation*}
\begin{xy}
\xymatrix{
0  \ar[r]  & M_4(C_0(Y\setminus Z))^{\beta} \ar[r] 
& M_4(C_0(Y))^{\beta} \ar[r] 
& M_4(C_0(Z))^{\beta} \ar[r]  & 0 }
\end{xy}
\end{equation*}
\end{lemma}

\begin{proof}
It suffices to show that $M_4(C_0(Y))^{\beta} \to 
M_4(C_0(Z))^{\beta}$ is surjective. 
The other assertions are easy to see. 

Take $f \in M_4(C_0(Z))^{\beta}$. 
Since $M_4(C_0(Y)) \to M_4(C_0(Z))$ is surjective, 
there exists $g\in M_4(C_0(Y))$ with $g|_{Z}=f$. 
Set $g_0 \in M_4(C_0(Y))$ by 
\[
g_0 \coloneqq \frac{1}{16}\sum_{i,j=1}^4 \beta_{i,j}(g).
\]
Then $g_0 \in M_4(C_0(Y))^{\beta}$ 
and $g_0|_{Z}=f$. 
This completes the proof. 
\end{proof}

We also use the following lemma many times. 

\begin{lemma}\label{Lem:isom1}
Let $Y$ be a locally compact subset of $\R P^3$ 
which is invariant under the action $\sigma$. 
Let $Z$ be a closed subset of $Y$ 
such that $Y = \bigcup_{i,j=1}^4 \sigma_{i,j}(Z)$
and that $\sigma_{i,j}(Z)\cap Z =\emptyset$ 
for $i,j=1,2,3,4$ with $(i,j)\neq (1,1)$. 
Then we have $M_4(C_0(Y))^{\beta} \cong M_4(C_0(Z))$. 
\end{lemma}

\begin{proof}
The restriction map $M_4(C_0(Y))^{\beta} \to M_4(C_0(Z))$ 
is an isomorphism because its inverse is given by 
\[
M_4(C_0(Z)) \ni f \mapsto \sum_{i,j=1}^4 \beta_{i,j}(f) \in 
M_4(C_0(Y))^{\beta}. \qedhere
\]
\end{proof}

Under the situation of the lemma above, 
$\pi\colon Z \to \pi(Z)=\pi(Y)$ is a homeomorphism. 
Hence we have 
$M_4(C_0(Y))^{\beta} \cong M_4(C_0(Z)) \cong M_4\big(C_0(\pi(Z))\big)
=M_4\big(C_0(\pi(Y))\big)$. 

The following lemma generalize Lemma~\ref{Lem:isom1}. 

\begin{lemma}
Let $G$ be a subgroup of $K \times K$. 
Let $Y$ be a locally compact subset of $\R P^3$ 
which is invariant under the action $\sigma$. 
Suppose that each point of $Y$ is fixed by 
$\sigma_{i,j}$ for all $(t_i,t_j) \in G$. 
Let $Z$ be a closed subset of $Y$ 
such that $Y = \bigcup_{i,j=1}^4 \sigma_{i,j}(Z)$
and that $\sigma_{i,j}(Z)\cap Z =\emptyset$ 
for $i,j=1,2,3,4$ with $(t_i,t_j) \not\in G$. 
Then we have $M_4(C_0(Y))^{\beta} \cong C_0(Z,D)$ 
where 
\[
D \coloneqq 
\{T \in M_4(\C) \mid \text{$\Ad U_{i,j}(T)=T$ for all 
$(t_i,t_j) \in G$}\}. 
\]
\end{lemma}

\begin{proof}
We have a restriction map $M_4(C_0(Y))^{\beta} \to C_0(Z,D)$ 
which is an isomorphism because its inverse is given by 
\[
C_0(Z,D) \ni f \mapsto \sum_{(i,j)\in I} \beta_{i,j}(f) \in 
M_4(C_0(Y))^{\beta},
\]
where an index set $I$ was chosen so that 
$\{(t_i,t_j)\in K\times K \mid (i,j)\in I\}$ 
becomes a complete representative of the quotient $(K\times K)/G$. 
\end{proof}

Under the situation of the lemma above, 
$\pi\colon Z \to \pi(Z)=\pi(Y)$ is a homeomorphism. 
Hence we have 
$M_4(C_0(Y))^{\beta} \cong C_0(Z,D) \cong C_0(\pi(Z),D)=C_0(\pi(Y),D)$.

\begin{definition}
We set 
$I \coloneqq M_4(C_0(\tO))^{\beta}$ and
$B \coloneqq M_4(C(\tF))^{\beta}$. 
\end{definition}

By Lemma~\ref{Lem:exseq} we get the short exact sequence 
\[
0 \longrightarrow I \longrightarrow A \longrightarrow B \longrightarrow 0,
\]
From this sequence, 
we get a six-term exact sequence 
\[
\xymatrix@C=25pt@R=20pt{
K_0(I) \ar[r] &  K_0(A) \ar[r] & K_0(B) \ar[d]^{\delta_0} \\
K_{1}(B) \ar[u]^{\delta_1} & \ar[l] K_{1}(A) & \ar[l] K_{1}(I).
}
\] 
From next section, 
we compute $K_i(B)$, $K_i(I)$ and $\delta_i$ for $i=0,1$. 
Consult \cite{RLL} for basics of K-theory.

\section{The Structure of the Quotient $B$}\label{Sec:StB}

\begin{definition}
For $i,j=2,3,4$, 
let $D_{i,j}$ be the fixed algebra of $\Ad U_{i,j}$ on $M_4(\C)$. 
\end{definition}

From the direct computation, 
we have the following. 

\begin{proposition}\label{Prop:M2M2}
For each $i,j =2,3,4$, 
$D_{i,j}$ is isomorphic to $M_2(\C) \oplus M_2(\C )$. 
More precisely, we have
\begin{align*}
&D_{2,2}=\left \{ 
\begin{pmatrix}
a &b &0 & 0\\
c & d & 0 & 0 \\
0 & 0 &e& f \\
0 & 0 & g &h
\end{pmatrix}
\right \}, &
&D_{2,3}=\left \{ 
\begin{pmatrix}
a &b &c & d\\
e & f & g & h \\
-h & g &f& -e \\
d & -c & -b &a
\end{pmatrix}
\right \}, \\
&D_{2,4}=\left \{ 
\begin{pmatrix}
a &b &c & d\\
e & f & g & h \\
c & d &a& b \\
g & h & e &f
\end{pmatrix}
\right \}, &
&D_{3,2}=\left \{ 
\begin{pmatrix}
a &b &c & d\\
e & f & g & h \\
h & g &f& e \\
d & c &b &a
\end{pmatrix}
\right \}, \\
& D_{3,3}=\left \{ 
\begin{pmatrix}
a &0 &b & 0\\
0 & c & 0 & d \\
e & 0 &f& 0 \\
0 & g &0 &h
\end{pmatrix}
\right \}, &
&D_{3,4}=\left \{ 
\begin{pmatrix}
a &b &c & d\\
b & a & -d & -c \\
e & f &g& h \\
-f & -e &h &g
\end{pmatrix}
\right \}, \\
&D_{4,2}=\left \{ 
\begin{pmatrix}
a &b &c & d\\
e & f & g & h \\
c & -d &a& -b \\
-g & h & -e &f
\end{pmatrix}
\right \}, &
&D_{4,3}=\left \{ 
\begin{pmatrix}
a &b &c & d\\
b & a & d & c \\
e & f &g& h \\
f & e & h &g
\end{pmatrix}
\right \}, \\
&D_{4,4}=\left \{ 
\begin{pmatrix}
a &0 &0 & b\\
0 & c & d & 0 \\
0 & e &f& 0 \\
g & 0 &0 &h
\end{pmatrix}
\right \}, 
\end{align*}
where $a,b,c,d,e,f,g,h$ run through $\C$. 
\end{proposition}

\begin{definition}
For each $i_2,i_3,i_4$ with $\{i_2,i_3,i_4  \}= \{2,3,4\}$,
define $D_{(i_2 i_3 i_4)} \subset \R P^3$ by
\begin{equation*}
D_{(i_2 i_3 i_4)} \coloneqq D_{i_2,2} \cap D_{i_3, 3} \cap D_{i_4,4}.
\end{equation*}
\end{definition}

\begin{proposition}\label{Prop:C4}
For each $i_2,i_3,i_4$ with $\{i_2,i_3,i_4  \}= \{2,3,4\}$,
we have 
\begin{equation*}
D_{(i_2 i_3 i_4)}= D_{i_2,2} \cap D_{i_3,3}= D_{i_2,2} \cap D_{i_4,4}= D_{i_3,3} \cap D_{i_4,4}, 
\end{equation*}
and $D_{(i_2 i_3 i_4)}$ is isomorphic to $\C^4$. 
More precisely, we have 
\begin{align*}
&D_{(234)}=\left \{ 
\begin{pmatrix}
a &0 &0 & 0\\
0 & b & 0 & 0 \\
0 & 0 &c& 0 \\
0 & 0 & 0 &d
\end{pmatrix}
\right \} &
&D_{(423)}=\left \{ 
\begin{pmatrix}
a &b &c & d\\
b & a & -d & -c \\
c & -d &a& -b \\
d & -c & -b &a
\end{pmatrix}
\right \} \\
&D_{(342)}=\left \{ 
\begin{pmatrix}
a &b &c & d\\
b & a & d & c \\
c & d &a& b \\
d & c & b &a
\end{pmatrix}
\right \} &
&D_{(243)}=\left \{ 
\begin{pmatrix}
a &b &0 & 0\\
b & a & 0 & 0 \\
0 & 0 &c& d \\
0 & 0 &d &c
\end{pmatrix}
\right \} \\
&D_{(432)}=\left \{ 
\begin{pmatrix}
a &0 &b & 0\\
0 & c & 0 & d \\
b & 0 &a& 0 \\
0 & d & 0 &c
\end{pmatrix}
\right \} &
&D_{(324)}=\left \{ 
\begin{pmatrix}
a &0 &0 & d\\
0 & b & c & 0 \\
0 & c &b& 0 \\
d & 0 & 0 &a
\end{pmatrix}
\right \} 
\end{align*}
where $a,b,c,d$ run through $\C$. 
\end{proposition}

\begin{definition}
We set 
$B^{\circ} \coloneqq M_4(C_0(\tF^{\circ}))^\beta$ and 
$B^{\text{\textbullet}} \coloneqq M_4(C(\tF^{\text{\textbullet}}))^\beta$. 
We also set 
$B^{\circ}_{i,j} \coloneqq M_4(C_0(\tF^{\circ}_{i,j}))^\beta$ 
for $i,j=2,3,4$ and 
$B_{(i_2i_3i_4)} \coloneqq M_4(C_0(\tF_{(i_2i_3i_4)}))^\beta$ 
for $i_2,i_3,i_4$ with $\{i_2,i_3,i_4\}= \{2,3,4\}$. 
\end{definition}

From the discussion up to here, 
we have the following proposition. 

\begin{proposition}
We have 
\begin{align*}
B^{\circ} &\cong \bigoplus_{i,j=2}^4 B^{\circ}_{i,j}, & 
B^{\text{\textbullet}} &\cong 
\bigoplus_{\{i_2,i_3,i_4\}=\{2,3,4\}} B_{(i_2i_3i_4)}. 
\end{align*}
We also have 
\begin{align*}
B^{\circ}_{i,j}&\cong C_0(F^{\circ}_{i,j},D_{i,j}) 
\cong C_0\big((0,1),M_2(\C)\oplus M_2(\C)\big),\\
\end{align*}
for $i,j=2,3,4$ and 
\begin{align*}
B_{(i_2i_3i_4)}&\cong 
C(F_{(i_2i_3i_4)},D_{(i_2i_3i_4)}) 
\cong \C^4
\end{align*}
for $i_2,i_3,i_4$ with $\{i_2,i_3,i_4\}= \{2,3,4\}$. 
\end{proposition}

From this proposition, we get 
\begin{align*}
B^{\circ} &\cong C_0\big((0,1),M_2(\C)\oplus M_2(\C)\big)^9
\cong C_0\big((0,1),M_2(\C))^{18}, & 
B^{\text{\textbullet}} &\cong (\C^4)^6\cong \C^{24}.
\end{align*}

\section{K-groups of the quotient $B$}\label{Sec:KB}

From the short exact sequence 
\[
0 \longrightarrow B^{\circ} \longrightarrow B 
\longrightarrow B^{\text{\textbullet}} \longrightarrow 0,
\]
we get a six-term exact sequence 
\[
\xymatrix@C=25pt@R=20pt{
0=K_0(B^{\circ}) \ar[r] &  K_0(B) \ar[r] & K_0(B^{\text{\textbullet}}) 
\cong \Z^{24} \ar[d]^{\delta} \\
0=K_{1}(B^{\text{\textbullet}}) \ar[u] & \ar[l] K_{1}(B) & 
\ar[l]  K_{1}(B^{\circ})\cong \Z^{18}.}
\] 

Next we compute $\delta\colon K_0(B^{\text{\textbullet}}) 
\to K_{1}(B^{\circ})$. 

\begin{proposition}\label{Prop:AabB}
Under the isomorphism $\varPhi\colon A(4) \to A$, 
the \Ca $A^{\text{ab}}(4)$ is canonically 
isomorphic to $B^{\text{\textbullet}}$ .
\end{proposition}

\begin{proof}
Since $B^{\text{\textbullet}} \cong \C^{24}$ is commutative, 
the surjection 
$A(4) \cong A \twoheadrightarrow B \twoheadrightarrow B^{\text{\textbullet}}$ 
factors through the surjection $A(4) \twoheadrightarrow A^{\text{ab}}(4)$. 
The induced surjection 
$A^{\text{ab}}(4) \twoheadrightarrow B^{\text{\textbullet}}$ is an isomorphism 
because $A^{\text{ab}}(4)\cong \C^{24}$. 
\end{proof}

For $i,j=1,2,3,4$, 
we denote the image of $P_{i,j} \in A$ under a surjection 
by the same symbol $P_{i,j}$. 
By Proposition~\ref{Prop:Aabp} and Proposition~\ref{Prop:AabB}, 
the 24 minimal projections of $B^{\text{\textbullet}}$ are 
\[
P_{(i_1i_2i_3i_4)} \coloneqq 
P_{i_1,1}P_{i_2,2}P_{i_3,3}P_{i_4,4} \in B^{\text{\textbullet}}
\]
for $(i_1i_2i_3i_4) \in \mathfrak{S}_4$. 

\begin{definition}
For $\sigma \in \mathfrak{S}_4$, we define 
$q_\sigma \coloneqq [P_{\sigma}]_0 \in K_0(B^{\text{\textbullet}})$. 
\end{definition}

Note that $\{q_\sigma\}_{\sigma \in \mathfrak{S}_4}$ is a basis of 
$K_0(B^{\text{\textbullet}}) \cong \Z^{24}$. 

\begin{proposition}
For each $i_2,i_3,i_4$ with $\{i_2,i_3,i_4  \}= \{2,3,4\}$,
the $4$ minimal projections of 
$\C^4 \cong B_{(i_2i_3i_4)} 
\subset B^{\text{\textbullet}}$ 
are $P_{\sigma t_k}$ for $k=1,2,3,4$ 
where $\sigma \coloneqq (1i_2i_3i_4)$. 
\end{proposition}

\begin{proof}
Take $i_2,i_3,i_4$ with $\{i_2,i_3,i_4  \}= \{2,3,4\}$. 
Since the 4 points in $\tF_{(i_2i_3i_4)}$ 
are fixed by $\sigma_{i_2,2}$, $\sigma_{i_3,3}$ and $\sigma_{i_4,4}$, 
we have $P_{k,l}=P_{t_{i_j}(k),t_{j}(l)}$ 
in $B_{(i_2i_3i_4)}$
for $k,l=1,2,3,4$ and $j=2,3,4$ by Lemma~\ref{Lem:PaPa}. 
More concretely we have 
\begin{align*}
&P_{1,1}=P_{i_2,2}=P_{i_3,3}=P_{i_4,4},\\
&P_{i_2,1}=P_{1,2}=P_{i_4,3}=P_{i_3,4},\\
&P_{i_3,1}=P_{i_4,2}=P_{1,3}=P_{i_2,4},\\
&P_{i_4,1}=P_{i_3,2}=P_{i_2,3}=P_{1,4} 
\end{align*}
in $B_{(i_2i_3i_4)}$. 
These four projections are mutually orthogonal, 
and their sum equals to $1$. 
Thus the 4 minimal projections of $B_{(i_2i_3i_4)}$ 
are $P_{(1i_2i_3i_4)}$, $P_{(i_21i_4i_3)}$, 
$P_{(i_3i_41i_2)}$ and $P_{(i_4i_3i_21)}$. 
\end{proof}

Take $i,j=2,3,4$, and fix them for a while. 
Let $(1m_2m_3m_4) \in \fS_4$ be the unique even permutation 
with $m_j=i$, and 
$(1n_2n_3n_4) \in \fS_4$ be the unique odd permutation 
with $n_j=i$. 
We set $\sigma=(1m_2m_3m_4)$ and $\tau = (1n_2n_3n_4)$. 
Then we have the following commutative diagram with exact rows; 
\begin{equation*}
\xymatrix@C=10pt@R=20pt{
0 \ar[r] & B^{\circ} \ar[r] \ar@{->>}[d]  & B \ar[r]  \ar@{->>}[d] 
& B^{\text{\textbullet}} \ar[r]  \ar@{->>}[d]   & 0 \\
0 \ar[r] & B^{\circ}_{i,j} \ar[r] & B_{i,j} \ar[r] 
& B_{(m_2 m_3 m_4)} \oplus 
B_{(n_2 n_3 n_4)} \ar[r] & 0.
}
\end{equation*}
By Lemma~\ref{Lem:PaPa}, 
we have $P_{k,l}=P_{t_i(k),t_j(l)}$ in $B_{i,j}$ 
for $k,l=1,2,3,4$. 
Let $\omega=(1342) \in \fS_4$. 
Note that 
we have $t_i(\omega(i))=\omega^2(i)$ and $t_i(\omega^2(i))=\omega(i)$. 
One can see that 
$B_{i,j}$ is a direct sum of two \CsA s $B_{i,j}^{\cap}$ and $B_{i,j}^{\cup}$
where $B_{i,j}^{\cap}$ is generated by 
\begin{align*}
P_{1,1}&=P_{i,j},& P_{1,j}&=P_{i,1},& 
P_{\omega(i),\omega(j)}&=P_{\omega^2(i),\omega^2(j)},& 
P_{\omega(i),\omega^2(j)}&=P_{\omega^2(i),\omega(j)}
\end{align*}
and $B_{i,j}^{\cup}$ is generated by 
\begin{align*}
P_{1,\omega(j)}&=P_{i,\omega^2(j)},& 
P_{1,\omega^2(j)}&=P_{i,\omega(j)},& 
P_{\omega(i),1}&=P_{\omega^2(i),j},& 
P_{\omega(i),j}&=P_{\omega^2(i),1}. 
\end{align*}
Note that $P_{1,1}+P_{1,j}=P_{\omega(i),\omega(j)}+P_{\omega(i),\omega^2(j)}$ 
is the unit of $B_{i,j}^{\cap}$, and 
$P_{1,\omega(j)}+P_{1,\omega^2(j)}=P_{\omega(i),1}+P_{\omega(i),j}$ 
is the unit of $B_{i,j}^{\cup}$. 
It turns out that both $B_{i,j}^{\cap}$ and $B_{i,j}^{\cup}$ 
are isomorphic to the universal unital \Ca 
generated by two projections, which is isomorphic to 
\[
\Big\{f\in C([0,1],M_2(\C))\ \Big|\ 
f(0)=\begin{pmatrix} * & 0\\ 0 & *\end{pmatrix}, 
f(1)=\begin{pmatrix} * & 0\\ 0 & *\end{pmatrix}\Big\}.
\]
This fact can be proved directly, but we do not prove it here 
because we do not need it. 
The image of $B_{i,j}^{\cap}$ under 
the surjection $B_{i,j} \to B_{(m_2 m_3 m_4)} \oplus B_{(n_2 n_3 n_4)}$
is $(\C p_{\sigma}+\C p_{\sigma t_j}) \oplus (\C p_{\tau}+\C p_{\tau t_j})$. 
Therefore, the image of $B_{i,j}^{\cup}$ under 
the surjection $B_{i,j} \to B_{(m_2 m_3 m_4)} \oplus B_{(n_2 n_3 n_4)}$
is $(\C p_{\sigma t_{\omega(j)}}+\C p_{\sigma t_{\omega^2(j)}}) \oplus 
(\C p_{\tau t_{\omega(j)}}+\C p_{\tau t_{\omega^2(j)}})$. 
We set $v_{i,j}^{\cap},v_{i,j}^{\cup} \in K_1(B^{\circ}_{i,j})$ 
by $v_{i,j}^{\cap} \coloneqq \delta'(q_\sigma)$ 
and $v_{i,j}^{\cup} \coloneqq \delta'(q_{\sigma t_{\omega(j)}})$ 
where 
\[
\delta'\colon K_0(B_{(m_2 m_3 m_4)} \oplus B_{(n_2 n_3 n_4)})
\to K_1(B^{\circ}_{i,j})
\]
is the exponential map. 
Then we have the following. 

\begin{lemma}\label{Lem:dq=v}
The set $\{v_{i,j}^{\cap},v_{i,j}^{\cup}\}$ is 
a generator of $K_1(B^{\circ}_{i,j}) \cong \Z^2$, 
and we have 
\begin{align*}
\delta'(q_\sigma)&=\delta'(q_{\sigma t_j})=v_{i,j}^{\cap}, & 
\delta'(q_{\sigma t_{\omega(j)}})&=\delta'(q_{\sigma t_{\omega^2(j)}})=v_{i,j}^{\cup}, \\
\delta'(q_\tau)&=\delta'(q_{\tau t_j})=-v_{i,j}^{\cap}, & 
\delta'(q_{\tau t_{\omega(j)}})&=\delta'(q_{\tau t_{\omega^2(j)}})=-v_{i,j}^{\cup}. \\
\end{align*}
\end{lemma}

\begin{proof}
Choose a closed interval $Z \subset \R P^3$ 
such that $\pi\colon Z \to F_{i,j}$ is a homeomorphism 
(see Figure~\ref{fig:EF} and the remark around it 
for an example of such a space). 
Let $z_0,z_1 \in Z$ be the point such that $\pi(z_0)=v_{(m_2 m_3 m_4)}$
$\pi(z_1)=v_{(n_2 n_3 n_4)}$. 
Then we have $B^{\circ}_{i,j} \cong C_0(Z\setminus\{z_0,z_1\},D_{i,j})$. 
Let $B_{i,j}'$ be the inverse image of $B_{(m_2 m_3 m_4)}$ 
under the surjection $B_{i,j} \to B_{(m_2 m_3 m_4)} \oplus B_{(n_2 n_3 n_4)}$. 
Then we have the following commutative diagram with exact rows; 
\[
\xymatrix@C=10pt@R=20pt{
0 \ar[r] & B^{\circ}_{i,j} \ar[r] \ar@{=}[d]  & B_{i,j}' \ar[r]  \ar@{->}[d] 
& B_{(m_2 m_3 m_4)} \ar[r]  \ar[d]   & 0 \\
0 \ar[r] & B^{\circ}_{i,j} \ar[r] & C_0(Z\setminus\{z_0\},D_{i,j}) \ar[r] 
& D_{i,j} \ar[r] & 0.
}
\]
Let us denote by $\varphi$ the homomorphism from $K_0(B_{(m_2 m_3 m_4)})$ 
to $K_0(D_{i,j})$ induced by 
the vertical map from $B_{(m_2 m_3 m_4)} \cong D_{(m_2 m_3 m_4)}$ to $D_{i,j}$.
Then $K_0(D_{i,j}) \cong \Z^2$ is spanned by 
$\varphi(q_\sigma)=\varphi(q_{\sigma t_j})$ and 
$\varphi(q_{\sigma t_{\omega(j)}})=\varphi(q_{\sigma t_{\omega^2(j)}})$. 
Since $K_l(C_0(Z\setminus\{z_0\},D_{i,j}))=0$ for $l=0,1$, 
$K_0(D_{i,j}) \to K_1(B^{\circ}_{i,j})$ is an isomorphism. 
This shows that $\{v_{i,j}^{\cap},v_{i,j}^{\cup}\}$ is 
a generator of $K_1(B^{\circ}_{i,j}) \cong \Z^2$. 
We also have $\delta'(q_\sigma)=\delta'(q_{\sigma t_j})$ 
and $\delta'(q_{\sigma t_{\omega(j)}})=\delta'(q_{\sigma t_{\omega^2(j)}})$. 
Similarly, we have 
$\delta'(q_\tau)=\delta'(q_{\tau t_j})$ and 
$\delta'(q_{\tau t_{\omega(j)}})=\delta'(q_{\tau t_{\omega^2(j)}})$. 

Since the image of the projection $P_{1,1} \in B_{i,j}$ under 
the surjection $B_{i,j} \to B_{(m_2 m_3 m_4)} \oplus B_{(n_2 n_3 n_4)}$ 
is $P_{\sigma}+P_{\tau}$, 
we have $\delta'(q_{\sigma}+q_{\tau})=0$. 
Hence $\delta'(q_\tau)=-v_{i,j}^{\cap}$. 
Similarly 
we have $\delta'(q_{\sigma t_{\omega(j)}}+q_{\tau t_{\omega(j)}})=0$ 
because the image of $P_{1,\omega(j)} \in B_{i,j}$ under 
the surjection $B_{i,j} \to B_{(m_2 m_3 m_4)} \oplus B_{(n_2 n_3 n_4)}$ 
is $P_{\sigma t_{\omega(j)}}+P_{\tau t_{\omega(j)}}$. 
We are done. 
\end{proof}

From these computation, we get the following proposition. 

\begin{proposition}
The exponential map $\delta\colon K_0(B^{\text{\textbullet}}) \to K_1(B^{\circ})$ is as Table \ref{table:exp_def}. 
\end{proposition}

\begin{table}[hbtp]
%\begin{center}
\caption{Computation of the exponential map $\delta$}\label{table:exp_def}
\begin{tabular}{|c||cc|cc|cc|cc|cc|cc|cc|cc|cc|} \hline 
& \multicolumn{2}{|c|}{2,2} & \multicolumn{2}{|c|}{3,3} & \multicolumn{2}{|c|}{4,4} & \multicolumn{2}{|c|}{4,3} & \multicolumn{2}{|c|}{2,4} & \multicolumn{2}{|c|}{3,2} & \multicolumn{2}{|c|}{3,4} & \multicolumn{2}{|c|}{4,2} & \multicolumn{2}{|c|}{2,3} \\ \hline
$q\diagdown v$ & $\cap$ & $\cup$ & $\cap$ & $\cup$ & $\cap$ & $\cup$ & $\cap$ & $\cup$ & $\cap$ & $\cup$ & $\cap$ & $\cup$ & $\cap$ & $\cup$ & $\cap$ & $\cup$ & $\cap$ & $\cup$  \\ \hline\hline
(1234) & 1 & 0 & 1 & 0 & 1 & 0 & 0 & 0 & 0 & 0 & 0 & 0 & 0 & 0 & 0 & 0 & 0 & 0  \\ \hline
(2143) & 1 & 0 & 0 & 1 & 0 & 1 & 0 & 0 & 0 & 0 & 0 & 0 & 0 & 0 & 0 & 0 & 0 & 0  \\ \hline
(3412) & 0 & 1 & 1 & 0 & 0 & 1 & 0 & 0 & 0 & 0 & 0 & 0 & 0 & 0 & 0 & 0 & 0 & 0  \\ \hline 
(4321) & 0 & 1 & 0 & 1 & 1 & 0 & 0 & 0 & 0 & 0 & 0 & 0 & 0 & 0 & 0 & 0 & 0 & 0  \\ \hline \hline
(1342) & 0 & 0 & 0 & 0 & 0 & 0 & 1 & 0 & 1 & 0 & 1 & 0 & 0 & 0 & 0 & 0 & 0 & 0  \\ \hline 
(2431) & 0 & 0 & 0 & 0 & 0 & 0 & 0 & 1 & 1 & 0 & 0 & 1 & 0 & 0 & 0 & 0 & 0 & 0  \\ \hline
(3124) & 0 & 0 & 0 & 0 & 0 & 0 & 0 & 1 & 0 & 1 & 1 & 0 & 0 & 0 & 0 & 0 & 0 & 0  \\ \hline 
(4213) & 0 & 0 & 0 & 0 & 0 & 0 & 1 & 0 & 0 & 1 & 0 & 1 & 0 & 0 & 0 & 0 & 0 & 0  \\ \hline \hline
(1423) & 0 & 0 & 0 & 0 & 0 & 0 & 0 & 0 & 0 & 0 & 0 & 0 & 1 & 0 & 1 & 0 & 1 & 0  \\ \hline
(2314) & 0 & 0 & 0 & 0 & 0 & 0 & 0 & 0 & 0 & 0 & 0 & 0 & 0 & 1 & 0 & 1 & 1 & 0  \\ \hline
(3241) & 0 & 0 & 0 & 0 & 0 & 0 & 0 & 0 & 0 & 0 & 0 & 0 & 1 & 0 & 0 & 1 & 0 & 1  \\ \hline
(4132) & 0 & 0 & 0 & 0 & 0 & 0 & 0 & 0 & 0 & 0 & 0 & 0 & 0 & 1 & 1 & 0 & 0 & 1  \\ \hline \hline
(1243) &\mo& 0 & 0 & 0 & 0 & 0 &\mo& 0 & 0 & 0 & 0 & 0 &\mo& 0 & 0 & 0 & 0 & 0  \\ \hline
(2134) &\mo& 0 & 0 & 0 & 0 & 0 & 0 &\mo& 0 & 0 & 0 & 0 & 0 &\mo& 0 & 0 & 0 & 0  \\ \hline
(3421) & 0 &\mo& 0 & 0 & 0 & 0 & 0 &\mo& 0 & 0 & 0 & 0 &\mo& 0 & 0 & 0 & 0 & 0  \\ \hline
(4312) & 0 &\mo& 0 & 0 & 0 & 0 &\mo& 0 & 0 & 0 & 0 & 0 & 0 &\mo& 0 & 0 & 0 & 0  \\ \hline \hline 
(1432) & 0 & 0 &\mo& 0 & 0 & 0 & 0 & 0 &\mo& 0 & 0 & 0 & 0 & 0 &\mo& 0 & 0 & 0  \\ \hline
(2341) & 0 & 0 & 0 &\mo& 0 & 0 & 0 & 0 &\mo& 0 & 0 & 0 & 0 & 0 & 0 &\mo& 0 & 0  \\ \hline
(3214) & 0 & 0 &\mo& 0 & 0 & 0 & 0 & 0 & 0 &\mo& 0 & 0 & 0 & 0 & 0 &\mo& 0 & 0  \\ \hline
(4123) & 0 & 0 & 0 &\mo& 0 & 0 & 0 & 0 & 0 &\mo& 0 & 0 & 0 & 0 &\mo& 0 & 0 & 0  \\ \hline \hline
(1324) & 0 & 0 & 0 & 0 &\mo& 0 & 0 & 0 & 0 & 0 &\mo& 0 & 0 & 0 & 0 & 0 &\mo& 0  \\ \hline
(2413) & 0 & 0 & 0 & 0 & 0 &\mo& 0 & 0 & 0 & 0 & 0 &\mo& 0 & 0 & 0 & 0 &\mo& 0  \\ \hline
(3142) & 0 & 0 & 0 & 0 & 0 &\mo& 0 & 0 & 0 & 0 &\mo& 0 & 0 & 0 & 0 & 0 & 0 &\mo \\ \hline
(4231) & 0 & 0 & 0 & 0 &\mo& 0 & 0 & 0 & 0 & 0 & 0 &\mo& 0 & 0 & 0 & 0 & 0 &\mo \\ \hline  
\end{tabular}
%\end{center}
\end{table}

We will see that $K_1(B) \cong \coker \delta$ is isomorphic 
to $\Z^4 \oplus \Z/2\Z$ in Proposition~\ref{Prop:KB}. 
This implies $K_0(B) \cong \ker \delta$ is isomorphic to $\Z^{10}$ 
because $\ker \delta$ is a free abelian group with dimension $24-18+4=10$. 
Below, we examine the generator of $K_0(B) \cong \ker \delta$. 

For $i,j=1,2,3,4$, 
we have 
\begin{align*}
P_{i,j}=P_{i,j}\sum_{k\neq i}\sum_{l=1}^nP_{k,l}
=\sum_{i=\sigma(j)}P_\sigma
\end{align*}
in $B$. 
Hence $[P_{i,j}]_0=\sum_{i=\sigma(j)}q_\sigma$ 
in $K_0(B)$. 

\begin{proposition}\label{Prop:kerdelta}
The group $\ker \delta$ 
is generated by $\{[P_{i,j}]_0\mid i,j=1,2,3,4\}$. 
\end{proposition}

\begin{proof}
It is straightforward to check that $[P_{i,j}]_0$ is in $\ker \delta$ 
for $i,j = 1,2,3,4$. 

Take $x \in \ker \delta$, and we will show that $x$ is in the subgroup 
generated by $\{[P_{i,j}]_0\mid i,j=1,2,3,4\}$. 
Write $x=\sum_{\sigma \in \fS_4}n_\sigma q_\sigma$ 
with $n_\sigma \in \Z$. 
Subtracting 
$n_{(4213)}[P_{2,2}]_0+n_{(4132)}[P_{1,2}]_0$
from $x$, 
we may assume $n_{(4213)}=n_{(4132)}=0$ 
without loss of generality. 
Subtracting 
$n_{(4312)}[P_{3,2}]_0+n_{(4123)}[P_{2,3}]_0+n_{(4231)}[P_{1,4}]_0$
from $x$, 
we may further assume $n_{(4312)}=n_{(4123)}=n_{(4231)}=0$ 
without loss of generality. 
Subtracting 
$n_{(2341)}[P_{2,1}]_0+n_{(3142)}[P_{3,1}]_0$
from $x$, 
we may further assume $n_{(2341)}=n_{(3142)}=0$ 
without loss of generality. 
Subtracting 
$n_{(2413)}[P_{4,2}]_0+n_{(3214)}[P_{4,4}]_0+n_{(1324)}[P_{1,1}]_0$
from $x$, 
we may further assume $n_{(2413)}=n_{(3214)}=n_{(1324)}=0$ 
without loss of generality. 
Now we will show $x=0$ using $x \in \ker \delta$. 

Since $n_{(3241)}+n_{(4132)}=n_{(3142)}+n_{(4231)}$, 
we have $n_{(3241)}=0$. 

Since $n_{(2314)}+n_{(3241)}=n_{(2341)}+n_{(3214)}$, 
we have $n_{(2314)}=0$. 

Since $n_{(1423)}+n_{(2314)}=n_{(1324)}+n_{(2413)}$, 
we have $n_{(1423)}=0$. 

Since $n_{(1423)}+n_{(4132)}=n_{(1432)}+n_{(4123)}$, 
we have $n_{(1432)}=0$. 

Since $n_{(3124)}+n_{(4213)}=n_{(3214)}+n_{(4123)}$, 
we have $n_{(3124)}=0$. 

Since $n_{(2431)}+n_{(4213)}=n_{(2413)}+n_{(4231)}$, 
we have $n_{(2431)}=0$. 

Since $n_{(1342)}+n_{(2431)}=n_{(1432)}+n_{(2341)}$, 
we have $n_{(1342)}=0$. 

Since $n_{(2314)}+n_{(4132)}=n_{(2134)}+n_{(4312)}$, 
we have $n_{(2134)}=0$. 

Since $n_{(2431)}+n_{(3124)}=n_{(2134)}+n_{(3421)}$, 
we have $n_{(3421)}=0$. 

Since $n_{(1423)}+n_{(3241)}=n_{(1243)}+n_{(3421)}$, 
we have $n_{(1243)}=0$. 

Since $n_{(1234)}+n_{(2143)}=n_{(1243)}+n_{(2134)}=0$, 
$n_{(1234)}+n_{(3412)}=n_{(1432)}+n_{(3214)}=0$ 
and $n_{(2143)}+n_{(3412)}=n_{(2413)}+n_{(3142)}=0$, 
we have $2 n_{(1234)}=0$. 
Hence $n_{(1234)}=0$. 
This implies $n_{(2143)}=n_{(3412)}=0$. 
Finally, since $n_{(1234)}+n_{(4321)}=n_{(1324)}+n_{(4231)}$, 
we have $n_{(4321)}=0$. 
We have shown that $x=0$. 
This completes the proof. 
\end{proof}

From Proposition~\ref{Prop:kerdelta} 
(or its proof), 
we see that $K_0(B) \cong \ker \delta$ is isomorphic to $\Z^{n}$ 
with $n \leq 10$. 
Note that the group generated by $\{[P_{i,j}]_0\mid i,j=1,2,3,4\}$ 
is in fact generated by $10$ elements
\[
[P_{1,1}]_0, [P_{1,2}]_0, [P_{1,3}]_0, [P_{1,4}]_0, [P_{2,1}]_0, 
[P_{2,2}]_0, [P_{2,3}]_0, [P_{3,1}]_0, [P_{3,2}]_0, [P_{3,3}]_0. 
\]
We will show that $K_0(B) \cong \ker \delta$ is isomorphic to $\Z^{10}$ 
in Proposition~\ref{Prop:KB}.

\begin{table}[hbtp]
%\begin{center}
\caption{Computation of $[P_{i,j}]_0$}\label{table:Pij}
\begin{tabular}{|c||cccc|cccc|cccc|cccc|} \hline
$i$ & \multicolumn{4}{|c|}{1} &  \multicolumn{4}{|c|}{2} &  \multicolumn{4}{|c|}{3} &  \multicolumn{4}{|c|}{4}  \\ \hline
$q\diagdown j$
       & 1 & 2 & 3 & 4 & 1 & 2 & 3 & 4 & 1 & 2 & 3 & 4 & 1 & 2 & 3 & 4  \\ \hline\hline
(1234) & 1 & 0 & 0 & 0 & 0 & 1 & 0 & 0 & 0 & 0 & 1 & 0 & 0 & 0 & 0 & 1  \\ \hline
(2143) & 0 & 1 & 0 & 0 & 1 & 0 & 0 & 0 & 0 & 0 & 0 & 1 & 0 & 0 & 1 & 0  \\ \hline
(3412) & 0 & 0 & 1 & 0 & 0 & 0 & 0 & 1 & 1 & 0 & 0 & 0 & 0 & 1 & 0 & 0  \\ \hline 
(4321) & 0 & 0 & 0 & 1 & 0 & 0 & 1 & 0 & 0 & 1 & 0 & 0 & 1 & 0 & 0 & 0  \\ \hline \hline
(1342) & 1 & 0 & 0 & 0 & 0 & 0 & 0 & 1 & 0 & 1 & 0 & 0 & 0 & 0 & 1 & 0  \\ \hline 
(2431) & 0 & 0 & 0 & 1 & 1 & 0 & 0 & 0 & 0 & 0 & 1 & 0 & 0 & 1 & 0 & 0  \\ \hline
(3124) & 0 & 1 & 0 & 0 & 0 & 0 & 1 & 0 & 1 & 0 & 0 & 0 & 0 & 0 & 0 & 1  \\ \hline 
(4213) & 0 & 0 & 1 & 0 & 0 & 1 & 0 & 0 & 0 & 0 & 0 & 1 & 1 & 0 & 0 & 0  \\ \hline \hline
(1423) & 1 & 0 & 0 & 0 & 0 & 0 & 1 & 0 & 0 & 0 & 0 & 1 & 0 & 1 & 0 & 0  \\ \hline
(2314) & 0 & 0 & 1 & 0 & 1 & 0 & 0 & 0 & 0 & 1 & 0 & 0 & 0 & 0 & 0 & 1  \\ \hline
(3241) & 0 & 0 & 0 & 1 & 0 & 1 & 0 & 0 & 1 & 0 & 0 & 0 & 0 & 0 & 1 & 0  \\ \hline
(4132) & 0 & 1 & 0 & 0 & 0 & 0 & 0 & 1 & 0 & 0 & 1 & 0 & 1 & 0 & 0 & 0  \\ \hline \hline
(1243) & 1 & 0 & 0 & 0 & 0 & 1 & 0 & 0 & 0 & 0 & 0 & 1 & 0 & 0 & 1 & 0  \\ \hline
(2134) & 0 & 1 & 0 & 0 & 1 & 0 & 0 & 0 & 0 & 0 & 1 & 0 & 0 & 0 & 0 & 1  \\ \hline
(3421) & 0 & 0 & 0 & 1 & 0 & 0 & 1 & 0 & 1 & 0 & 0 & 0 & 0 & 1 & 0 & 0  \\ \hline
(4312) & 0 & 0 & 1 & 0 & 0 & 0 & 0 & 1 & 0 & 1 & 0 & 0 & 1 & 0 & 0 & 0  \\ \hline \hline 
(1432) & 1 & 0 & 0 & 0 & 0 & 0 & 0 & 1 & 0 & 0 & 1 & 0 & 0 & 1 & 0 & 0  \\ \hline
(2341) & 0 & 0 & 0 & 1 & 1 & 0 & 0 & 0 & 0 & 1 & 0 & 0 & 0 & 0 & 1 & 0  \\ \hline
(3214) & 0 & 0 & 1 & 0 & 0 & 1 & 0 & 0 & 1 & 0 & 0 & 0 & 0 & 0 & 0 & 1  \\ \hline
(4123) & 0 & 1 & 0 & 0 & 0 & 0 & 1 & 0 & 0 & 0 & 0 & 1 & 1 & 0 & 0 & 0  \\ \hline \hline
(1324) & 1 & 0 & 0 & 0 & 0 & 0 & 1 & 0 & 0 & 1 & 0 & 0 & 0 & 0 & 0 & 1  \\ \hline
(2413) & 0 & 0 & 1 & 0 & 1 & 0 & 0 & 0 & 0 & 0 & 0 & 1 & 0 & 1 & 0 & 0  \\ \hline
(3142) & 0 & 1 & 0 & 0 & 0 & 0 & 0 & 1 & 1 & 0 & 0 & 0 & 0 & 0 & 1 & 0  \\ \hline
(4231) & 0 & 0 & 0 & 1 & 0 & 1 & 0 & 0 & 0 & 0 & 1 & 0 & 1 & 0 & 0 & 0  \\ \hline  
\end{tabular}
%\end{center}
\end{table}

The positive cone $K_0(B^{\text{\textbullet}})_+$ of 
$K_0(B^{\text{\textbullet}})$ is the set of sums of 
$q_\sigma$'s. 
In other words, we have 
\[
K_0(B^{\text{\textbullet}})_+ = 
\Big\{\sum_{\sigma \in \fS_4}n_\sigma q_\sigma\ \Big|\ 
n_\sigma = 0,1,2,\ldots \Big\}
\]

\begin{proposition}\label{Prop:poscone}
The intersection $K_0(B^{\text{\textbullet}})_+ \cap \ker \delta$ 
is the set of sums of $[P_{i,j}]_0$'s. 
\end{proposition}

\begin{proof}
It is clear that $[P_{i,j}]_0$ is in 
$K_0(B^{\text{\textbullet}})_+ \cap \ker \delta$ 
for $i,j=1,2,3,4$. 
Thus the set of sums of $[P_{i,j}]_0$'s 
is contained in $K_0(B^{\text{\textbullet}})_+ \cap \ker \delta$. 

Take $x \in K_0(B^{\text{\textbullet}})_+ \cap \ker \delta$. 
By Proposition~\ref{Prop:kerdelta}, 
we have $n_{i,j} \in \Z$ for $i,j=1,2,3,4$ such that 
$x=\sum_{i,j=1}^4 n_{i,j}[P_{i,j}]_0$. 
We set $n \coloneqq \sum_{n_{i,j}<0}(-n_{i,j})$. 
If $n=0$, then $x$ is in the set of sums of $[P_{i,j}]_0$'s. 
If $n >0$, then we will show that there exists 
$n'_{i,j} \in \Z$ for $i,j=1,2,3,4$ such that 
$x=\sum_{i,j=1}^4 n'_{i,j}[P_{i,j}]_0$ and that 
$n' \coloneqq \sum_{n'_{i,j}<0}(-n'_{i,j})$ satisfies 
$0 \leq n' <n$. 
Repeating this argument at most $n$ times, 
we will find $n''_{i,j} \in \Z$ for $i,j=1,2,3,4$ such that 
$x=\sum_{i,j=1}^4 n''_{i,j}[P_{i,j}]_0$ and that 
$n'' \coloneqq \sum_{n''_{i,j}<0}(-n''_{i,j})$ satisfies 
$n''=0$. 
This shows that $x$ is in the set of sums of $[P_{i,j}]_0$'s. 

Since $n>0$ we have $i_0,j_0\in \{1,2,3,4\}$ 
such that $n_{i_0,j_0}<0$. 
To simplify the notation, we assume $i_0=3$ and $j_0=1$. 
The other $15$ cases can be shown similarly. 
Since $x \in K_0(B^{\text{\textbullet}})_+$, 
the coefficient of $v_\sigma$ is non-negative 
for all $\sigma \in \fS_4$. 
In particular, so is for $\sigma \in \fS_4$ 
with $i_0=\sigma(j_0)$. 
Since the coefficient of $v_{(3,1,2,4)}$ is non-negative 
we have $n_{3,1}+n_{1,2}+n_{2,3}+n_{4,4}\geq 0$. 
Since $n_{3,1}<0$, 
we have $n_{1,2}+n_{2,3}+n_{4,4} > 0$. 
Hence either $n_{1,2}$, $n_{2,3}$ or $n_{4,4}$ 
is positive. 
Similarly, since the coefficients of 
\[
v_{(3,1,4,2)}, v_{(3,2,1,4)}, v_{(3,2,4,1)}, v_{(3,4,1,2)}, v_{(3,4,2,1)} 
\] 
are non-negative, 
we obtain that either $n_{1,2}$, $n_{4,3}$ or $n_{2,4}$ 
is positive etc. 
Then by Lemma~\ref{Lem:3x3} below 
we have either 
\begin{enumerate}
\rom
\item $n_{i_1,2}$ $n_{i_1,3}$ and $n_{i_1,4}$ 
are positive for some $i_1 \in \{1,2,4\}$, 
\item $n_{1,j_1}$ $n_{2,j_1}$ and $n_{4,j_1}$ 
are positive for some $j_1 \in \{2,3,4\}$, 
or 
\item $n_{i_1,j_1}$, $n_{i_1,j_2}$, $n_{i_2,j_1}$ and $n_{i_2,j_2}$ 
are positive for some distinct $i_1,i_2 \in \{1,2,4\}$ 
and distinct $j_1,j_2 \in \{2,3,4\}$. 
\end{enumerate}
In the case (i), 
we set $n'_{i,j}$ by 
\[
n'_{i,j}=
\begin{cases}
n_{i,j}+1& \text{for $i\in \{1,2,3,4\}\setminus\{i_1\}$ and $j=1$,}\\ 
n_{i,j}-1& \text{for $i=i_1$ and $j=2,3,4$} \\
n_{i,j}& \text{otherwise.}
\end{cases}
\]
Then since $n'_{3,1}=n_{3,1}+1$, 
$n' \coloneqq \sum_{n'_{i,j}<0}(-n'_{i,j})$ satisfies 
$0 \leq n' <n$. 
We also have $x=\sum_{i,j=1}^4 n'_{i,j}[P_{i,j}]_0$ 
because 
$\sum_{i=1}^4[P_{i,1}]_0=\sum_{j=1}^4[P_{i_1,j}]_0$. 
In the case (ii), 
we get the same conclusion for $n'_{i,j}$ defined by 
\[
n'_{i,j}=
\begin{cases}
n_{i,j}+1& \text{for $i=3$ and $j\in \{1,2,3,4\}\setminus\{j_1\}$,}\\ 
n_{i,j}-1& \text{for $i=1,2,4$ and $j=j_1$} \\
n_{i,j}& \text{otherwise.}
\end{cases}
\]
In the case (iii), 
we define $n'_{i,j}$ by 
\[
n'_{i,j}=
\begin{cases}
n_{i,j}+1& \text{for $i\in \{1,2,3,4\}\setminus\{i_1,i_2\}$ 
and $j\in \{1,2,3,4\}\setminus\{j_1,j_2\}$,}\\ 
n_{i,j}-1& \text{for $i=i_1,i_2$ and $j=j_1,j_2$} \\
n_{i,j}& \text{otherwise.}
\end{cases}
\]
Since $n'_{3,1}=n_{3,1}+1$, 
$n' \coloneqq \sum_{n'_{i,j}<0}(-n'_{i,j})$ satisfies 
$0 \leq n' <n$. 
We also have $x=\sum_{i,j=1}^4 n'_{i,j}[P_{i,j}]_0$ 
because 
\[
\sum_{i=1}^4[P_{i,j_1}]_0
+\sum_{i=1}^4[P_{i,j_2}]_0
=\sum_{j=1}^4[P_{i_3,j}]_0+\sum_{j=1}^4[P_{i_4,j}]_0. 
\]
where $\{i_3,i_4\}=\{1,2,3,4\}\setminus\{i_1,i_2\}$. 
This completes the proof. 
\end{proof}

\begin{lemma}\label{Lem:3x3}
Let $a,b,c$ and $d,e,f$ are distinct three numbers, respectively. 
Suppose $n_{i,j}\in \Z$ for $i=a,b,c$ and $j=d,e,f$ 
satisfy that either $n_{\omega(d),d}$, 
$n_{\omega(e),e}$ or $n_{\omega(f),f}$ is positive 
for all bijection $\omega\colon \{d,e,f\} \to \{a,b,c\}$. 
Then we have either 
\begin{enumerate}
\rom
\item $n_{i_1,d}$ $n_{i_1,e}$ and $n_{i_1,f}$ 
are positive for some $i_1 \in \{a,b,c\}$, 
\item $n_{a,j_1}$ $n_{b,j_1}$ and $n_{c,j_1}$ 
are positive for some $j_1 \in \{d,e,f\}$, 
or 
\item $n_{i_1,j_1}$, $n_{i_1,j_2}$, $n_{i_2,j_1}$ and $n_{i_2,j_2}$ 
are positive for some distinct $i_1,i_2 \in \{a,b,c\}$ 
and distinct $j_1,j_2 \in \{d,e,f\}$. 
\end{enumerate}
\end{lemma}

\begin{proof}
To the contrary, assume that the conclusion does not hold. 
Then for $j=d,e,f$, 
either $n_{a,j}$, $n_{b,j}$ or $n_{c,j}$ 
is non-positive. 
Thus we obtain a map $\omega\colon \{d,e,f\} \to \{a,b,c\}$ 
such that $n_{\omega(j),j}$ is non-positive for $j=d,e,f$. 
If the cardinality of the image of $\omega$ is three, 
then $\omega$ is a bijection and it contradicts the assumption. 
If the cardinality of the image of $\omega$ is two, 
let $i_1$ be the element in $\{a,b,c\}$ which is not in the image of $\omega$. 
Then we have either $n_{i_1,d}$ $n_{i_1,e}$ or $n_{i_1,f}$ 
is non-positive. 
Let $j_1 \in \{d,e,f\}$ be an element such that $n_{i_1,j_1}$ 
is non-positive. 
If the cardinality of $\omega^{-1}(\omega(j_1))$ is two, 
we get a bijection $\omega'\colon \{d,e,f\} \to \{a,b,c\}$ 
such that $n_{\omega(d),d}$, 
$n_{\omega(e),e}$ and $n_{\omega(f),f}$ are non-positive. 
This is a contradiction. 
If the cardinality of $\omega^{-1}(\omega(j_1))$ is one, 
we have either $n_{i_1,j_2}$, $n_{i_1,j_3}$, $n_{i_2,j_2}$ or $n_{i_2,j_3}$ 
is non-positive 
where $i_2=\omega(j_1)$ and $\{j_2,j_3\}=\{d,e,f\}\setminus \{j_1\}$. 
In this case, 
we can find a bijection $\omega'\colon \{d,e,f\} \to \{a,b,c\}$ 
such that $n_{\omega(d),d}$, 
$n_{\omega(e),e}$ and $n_{\omega(f),f}$ are non-positive. 
This is a contradiction. 
Finally, 
if the cardinality of the image of $\omega$ is one, 
let $i_1$ be the unique element of the image of $\omega$, 
and $i_2$ and $i_3$ be the other two elements in $\{a,b,c\}$. 
We have $j_2,j_3 \in \{d,e,f\}$ such that 
$n_{i_2,j_2}$ and $n_{i_3,j_3}$ are non-positive. 
If $j_2 \neq j_3$, then we can find a bijection $\omega'\colon \{d,e,f\} \to \{a,b,c\}$ 
such that $n_{\omega(d),d}$, 
$n_{\omega(e),e}$ and $n_{\omega(f),f}$ are non-positive. 
This is a contradiction. 
If $j_2=j_3$, then 
we have either $n_{i_2,j_1}$, $n_{i_2,j'_1}$, $n_{i_3,j'_1}$ or $n_{i_3,j_1}$ 
is non-positive 
where $\{j_1,j'_1\}=\{d,e,f\}\setminus \{j_2\}$. 
In this case, we can find a bijection $\omega'\colon \{d,e,f\} \to \{a,b,c\}$ 
such that $n_{\omega(d),d}$, 
$n_{\omega(e),e}$ and $n_{\omega(f),f}$ are non-positive. 
This is a contradiction.
We are done. 
\end{proof}

\section{The Structure of the Ideal $I$}\label{Sec:StI}

\begin{definition}
Define a subspace $V$ of $\R P^3$ by 
\begin{equation*}
V\coloneqq \{ [a_1,a_2, a_3, a_4] \in \R P^3 \mid a_1,a_2,a_3 > |a_4| \}.
\end{equation*}
\end{definition}

The next proposition gives us a motivation to compute 
the subspace $V$ and its closure $\overline{V}$ in $\R P^3$.

\begin{proposition}
 We have the following facts.
\begin{enumerate}
\rom
\item For each $i,j=1,2,3,4$ with $(i,j) \neq (1,1)$, 
we have $\sigma_{i,j}(V) \cap V = \emptyset$ 
\item 
The restriction of $\pi$ to $V$ is a homeomorphism onto $\pi(V) \subset X$ .
\item $\overline{V}=\{ [a_1,a_2, a_3, a_4] \in \R P^3 \mid a_1,a_2,a_3 \geq |a_4| \}$ and $\pi(\overline{V})=X$.
\end{enumerate}
\end{proposition}

\begin{proof}
(i) and (iii) can be checked directly, and 
(ii) follows from (i). 
\end{proof}

In the next proposition, 
when we write $[a_1,a_2,a_3,a_4] \in \overline{V}$, 
we mean $(a_1,a_2,a_3,a_4)$ satisfies $a_1,a_2,a_3 \geq |a_4|$. 

\begin{proposition}\label{Prop:314}
The map 
\[
h\colon \overline{V} \ni [a_1,a_2,a_3,a_4] 
\mapsto (3a_1^2+a_4^2+4a_4|a_4|,3a_2^2+a_4^2+4a_4|a_4|,3a_3^2+a_4^2+4a_4|a_4|)
\in \R^3
\]
is a homeomorphism onto the hexahedron 
whose 6 faces are isosceles right triangles 
and whose vertices are 
$(0,0,0)$, $(3,0,0)$, $(0,3,0)$, $(0,0,3)$ and $(2,2,2)$. 
This map sends $V$ onto the interior of the hexahedron. 
\end{proposition}

\begin{proof}
First note that we have $|a_4| \leq 1/2$ 
for $[a_1,a_2,a_3,a_4] \in \overline{V}$. 
When $|a_4| = 1/2$, we have $a_1=a_2=a_3=1/2$. 
We have $h([1/2,1/2,1/2,1/2])=(2,2,2)$ and $h([1/2,1/2,1/2,-1/2])=(0,0,0)$. 
When $|a_4| = 0$, we have $a_1,a_2,a_3 \geq 0$ and $a_1^2+a_2^2+a_3^2=1$. 
Thus 
\[
\{h([a_1,a_2,a_3,0]) \mid [a_1,a_2,a_3,0] \in \overline{V}\} 
=\{(x,y,z) \in \R^3 \mid x,y,z\geq 0,\ x+y+z=3\}
\]
which is the equilateral triangle 
whose vertices are $(3,0,0)$, $(0,3,0)$ and $(0,0,3)$. 
For each $t$ with $-1/2<t<0$, we have 
\begin{align*}
\{h([a_1,a_2,a_3,t]) &\mid [a_1,a_2,a_3,t] \in \overline{V}\} \\
&=\{(x,y,z) \in \R^3 \mid x,y,z\geq 0,\ x+y+z=3(1-4t^2)\}
\end{align*}
which is the equilateral triangle 
whose vertices are $(3(1-4t^2),0,0)$, $(0,3(1-4t^2),0)$ and $(0,0,3(1-4t^2))$. 
Thus 
\[
\{h([a_1,a_2,a_3,a_4]) \mid [a_1,a_2,a_3,a_4] \in \overline{V}, a_4 \leq 0\} 
\]
is the tetrahedron 
whose vertices are $(0,0,0)$, $(3,0,0)$, $(0,3,0)$ and $(0,0,3)$. 
Note that for each $[a_1,a_2,a_3,a_4] \in \overline{V}$ with $a_4 \geq 0$, 
the point $h([a_1,a_2,a_3,a_4])$ is the reflection point of 
$h([a_1,a_2,a_3,-a_4])$ with respect to the plane $x+y+z=3$ 
because the vector $(8a_4^2,8a_4^2,8a_4^2)$ 
is orthogonal to the plane $x+y+z=3$
and the point $(3a_1^2+a_4^2,3a_2^2+a_4^2,3a_3^2+a_4^2)$ 
is on the plane $x+y+z=3$. 
Thus 
\[
\{h([a_1,a_2,a_3,a_4]) \mid [a_1,a_2,a_3,a_4] \in \overline{V}, a_4 \geq 0\} 
\]
is the reflection of the tetrahedron above 
with respect to the plane $x+y+z=3$, 
which in turn is the tetrahedron 
whose vertices are $(3,0,0)$, $(0,3,0)$, $(0,0,3)$ and $(2,2,2)$. 
From the discussion above, 
we see that $h$ is injective. 
Therefore we see that $h$ is a homeomorphism from $\overline{V}$ 
onto the hexahedron whose vertices are 
$(0,0,0)$, $(3,0,0)$, $(0,3,0)$, $(0,0,3)$ and $(2,2,2)$. 
We can also see that 
the map $h$ sends $V$ onto the interior of the hexahedron. 
\end{proof}

\begin{figure}[hbtp]
\caption{$\overline{V}$}\label{fig:V}
\begin{center}
\input{V.tpc}
\end{center}
\end{figure}

\begin{definition}
Define $O_0 \coloneqq \pi(V) \subset O$. 
\end{definition}

By Proposition~\ref{Prop:314}, $O_0 \cong V$ is homeomorphic to $\R^3$. 

\begin{definition}
We set $E \coloneqq \tF \cap \overline{V}$ and 
$E_{i,j} \coloneqq \tF_{i,j} \cap \overline{V}$ for $i,j=2,3,4$. 
\end{definition}

We have $E=\bigcup_{i,j=2}^4 E_{i,j}$. 
For $i,j=2,3,4$ with $i \neq j$, 
the map $\pi\colon E_{i,j} \to F_{i,j}$ is a homeomorphism. 
For $i=2,3,4$ 
the map $\pi\colon E_{i,i} \to F_{i,i}$ is a $2$-to-$1$ map 
except the middle point. 

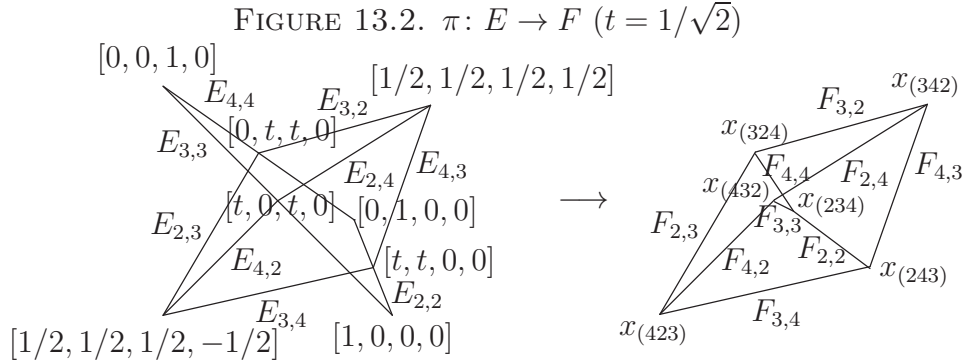
\begin{figure}[hbtp]
\caption{$\pi\colon E\to F$ ($t=1/\sqrt{2}$)}\label{fig:EF}
\begin{center}
\input{E.tpc}
\end{center}
\end{figure}

We have 
\begin{align*}
&E_{2,2}= \{[a,b,0,0] \in \overline{V} \mid a,b \geq 0,\ a^2 +b^2=1\}, \\
&E_{2,3}= \{ [a,b,b,-a] \in \overline{V} \mid 0 \leq a\leq b,\ 2(a^2 +b^2)=1  \}, \\
&E_{2,4}= \{[a,b,a,b]  \in \overline{V} \mid 0 \leq b\leq a,\ 2(a^2 +b^2)=1  \}, \\
&E_{3,2}= \{[a,b,b,a]  \in \overline{V}  \mid 0 \leq a\leq b,\ 2(a^2 +b^2)=1  \}, \\
&E_{3,3}= \{[a,0,b,0]  \in \overline{V} \mid a,b \geq 0,\ a^2 +b^2=1\}, \\
&E_{3,4}= \{[a,a,b,-b] \in \overline{V} \mid 0 \leq b\leq a,\ 2(a^2 +b^2)=1  \}, \\
&E_{4,2}= \{[a,b,a,-b]  \in \overline{V} \mid 0 \leq b\leq a,\ 2(a^2 +b^2)=1  \}, \\
&E_{4,3}= \{[a,a,b,b]  \in \overline{V} \mid 0 \leq b\leq a,\ 2(a^2 +b^2)=1   \}, \\
&E_{4,4}= \{ [0,a,b,0]  \in \overline{V} \mid a,b \geq 0,\ a^2 +b^2=1\}.
\end{align*}

\begin{definition}
We set $R^+_x, R^+_y, R^+_z, R^-_x, R^-_y, R^-_z \subset \overline{V}$ by 
\begin{align*}
R^\pm_x&\coloneqq \{[\sqrt{1-3t^2},t,t,\pm t] \in \overline{V} \mid 0<t<1/2\}\\
R^\pm_y&\coloneqq \{[t,\sqrt{1-3t^2},t,\pm t] \in \overline{V} \mid 0<t<1/2\}\\
R^\pm_z&\coloneqq \{[t,t,\sqrt{1-3t^2},\pm t] \in \overline{V} \mid 0<t<1/2\}
\end{align*}
\end{definition}

We see that $R^+_x \cup R^+_y \cup R^+_z \cup R^-_x \cup R^-_y \cup R^-_z$ 
is the space obtained by subtracting $E$ from the ``edges'' of $\overline{V}$. 

\begin{definition}
We set $R^+,R^- \subset O$ by 
\begin{align*}
R^\pm &\coloneqq \pi(R^\pm_x) =\pi(R^\pm_y)=\pi(R^\pm_z)
\end{align*}
\end{definition}

Note that $\pi$ induces a homeomorphism 
from $R^\pm_x$ (or $R^\pm_y$, $R^\pm_z$) to $R^\pm$. 
Hence both $R^+$ and $R^-$ are homeomorphic to $\R$. 

\begin{definition}
We set 
\begin{align*}
\hT_{2,3} &\coloneqq \{[t,a,b,-t] \in \overline{V} \mid 0<t<1/2,\ a,b>t,\ 
a^2+b^2=1-2t^2\},\\
\hT_{3,4} &\coloneqq \{[a,b,t,-t] \in \overline{V} \mid 0<t<1/2,\ a,b>t,\ 
a^2+b^2=1-2t^2\},\\
\hT_{4,2} &\coloneqq \{[b,t,a,-t] \in \overline{V} \mid 0<t<1/2,\ a,b>t,\ 
a^2+b^2=1-2t^2\},\\
\hT_{3,2} &\coloneqq \{[t,a,b,t] \in \overline{V} \mid 0<t<1/2,\ a,b>t,\ 
a^2+b^2=1-2t^2\},\\
\hT_{4,3} &\coloneqq \{[a,b,t,t] \in \overline{V} \mid 0<t<1/2,\ a,b>t,\ 
a^2+b^2=1-2t^2\},\\
\hT_{2,4} &\coloneqq \{[b,t,a,t] \in \overline{V} \mid 0<t<1/2,\ a,b>t,\ 
a^2+b^2=1-2t^2\}.
\end{align*}
\end{definition}

These 6 spaces are the interiors of the 6 ``faces'' of $\overline{V}$. 

\begin{definition}
We set 
\begin{align*}
\hT_{2,3}^r &\coloneqq \{[t,a,b,-t] \in \hT_{2,3}\mid a>b\}, & 
\hT_{2,3}^l &\coloneqq \{[t,a,b,-t] \in \hT_{2,3}\mid a<b\} \\
\hT_{3,4}^r &\coloneqq \{[a,b,t,-t] \in \hT_{3,4}\mid a>b\}, & 
\hT_{3,4}^l &\coloneqq \{[a,b,t,-t] \in \hT_{3,4}\mid a<b\} \\
\hT_{4,2}^r &\coloneqq \{[b,t,a,-t] \in \hT_{4,2}\mid a>b\}, & 
\hT_{4,2}^l &\coloneqq \{[b,t,a,-t] \in \hT_{4,2}\mid a<b\} \\
\hT_{3,2}^r &\coloneqq \{[t,a,b,t] \in \hT_{3,2}\mid a>b\}, & 
\hT_{3,2}^l &\coloneqq \{[t,a,b,t] \in \hT_{3,2}\mid a<b\} \\
\hT_{4,3}^r &\coloneqq \{[a,b,t,t] \in \hT_{4,3}\mid a>b\}, & 
\hT_{4,3}^l &\coloneqq \{[a,b,t,t] \in \hT_{4,3}\mid a<b\} \\
\hT_{2,4}^r &\coloneqq \{[b,t,a,t] \in \hT_{2,4}\mid a>b\}, & 
\hT_{2,4}^l &\coloneqq \{[b,t,a,t] \in \hT_{2,4}\mid a<b\}. 
\end{align*}
\end{definition}

For $i,j=2,3,4$ with $i\neq j$, 
the set $\hT_{i,j}\setminus (\hT_{i,j}^r \cup \hT_{i,j}^l)$ 
is the interior of $E_{i,j}$. 

\begin{definition}
For $i,j=2,3,4$ with $i\neq j$, 
we set 
\begin{align*}
T_{i,j} \coloneqq \pi(\hT_{i,j}^r) = \pi(\hT_{i,j}^l). 
\end{align*}
\end{definition}

Note that $\pi$ induces a homeomorphism 
from $\hT_{i,j}^r$ (or $\hT_{i,j}^l$) to $T_{i,j}$. 
Hence $T_{i,j}$ is homeomorphic to $\R^2$. 

The space $O$ is a disjoint union (as a set) of 
\[
O_0,\ T_{2,3},\ T_{3,4},\ T_{4,2},\ R^-,\ T_{3,2},\ T_{4,3},\ T_{2,4},\ R^+. 
\]
We use these spaces to compute the K-groups of $I =M_4(C_0(\tO))^{\beta}$.

\section{K-groups of the ideal $I$}\label{Sec:KI}

\begin{definition}
We set $I_0 \coloneqq M_4\big(C_0(\pi^{-1}(O_0))\big)^{\beta}$ 
and $I^{\star} \coloneqq M_4\big(C_0(\pi^{-1}(O\setminus O_0))\big)^{\beta}$. 
\end{definition}

We have a short exact sequence
\begin{equation*}
0 \longrightarrow I_0 \longrightarrow I \longrightarrow 
I^{\star} \longrightarrow 0.
\end{equation*}
We have $I_0 \cong M_4(C_0(V)) \cong M_4(C_0(O_0)) \cong M_4(C_0(\R^3))$. 

\begin{definition}
We set $T \coloneqq 
T_{2,3} \cup T_{3,4} \cup T_{4,2} \cup T_{3,2} \cup T_{4,3} \cup T_{2,4}$ 
and $R \coloneqq R^- \cup R^+$. 
We set $I^{\circ} \coloneqq M_4\big(C_0(\pi^{-1}(T))\big)^{\beta}$ 
and $I^{\text{\textbullet}} \coloneqq M_4\big(C_0(\pi^{-1}(R))\big)^{\beta}$. 
\end{definition}

We have $I^{\circ} \cong M_4(C_0(T)) \cong \bigoplus_{i,j} M_4(C_0(T_{i,j}))$
and 
\[
I^{\text{\textbullet}} \cong M_4(C_0(R)) 
\cong M_4(C_0(R^-))\oplus M_4(C_0(R^+)). 
\]
We have a short exact sequence
\begin{equation*}
0 \longrightarrow I^{\circ} \longrightarrow I^{\star} \longrightarrow 
I^{\text{\textbullet}} \longrightarrow 0.
\end{equation*}
This induces a six-term exact sequence 
\begin{equation*}
\xymatrix@C=10pt@R=20pt{
\Z^6 \cong K_0(I^\circ) \ar[r] &  K_0(I^{\star}) \ar[r] & K_0(I^{\text{\textbullet}}) \ar[d] \ar@{=}[r] & 0\\
\Z^2 \cong K_{1}( I^{\text{\textbullet}} ) \ar[u] & \ar[l] K_{1}(I^{\star}) & \ar[l] K_{1}(I^\circ) \ar@{=}[r] & 0.
}
\end{equation*}
We set $r^- \in K_1\big(M_4(C_0(R^-)) \big)$ 
and $r^+ \in K_1\big(M_4(C_0(R^+))\big)$ 
to be the images of $v_{(1234)} \in K_0(B_{(234)})\subset K_0(B^{\text{\textbullet}})$ 
under the exponential maps 
coming from the exact sequences 
\begin{equation*}
0 \longrightarrow M_4(C_0(R^{\pm})) \longrightarrow 
M_4\big(C_0(\pi^{-1}(R^{\pm}\cup\{x_{(234)}\}))\big)^\beta 
\longrightarrow B_{(234)} \longrightarrow 0.
\end{equation*}
Then similarly as the proof of Lemma~\ref{Lem:dq=v}, 
we see that $r^-$ and $r^+$ are the generators of 
$K_1\big(M_4(C_0(R^-))\big) \cong \Z$ 
and $K_1\big(M_4(C_0(R^+)) \big) \cong \Z$, respectively. 

Let $\omega=(1342) \in \fS_4$. 
For $i=2,3,4$, 
we set $w_{i,\omega(i)} \in K_0\big(M_4(C_0(T_{i,\omega(i)}))\big)$ 
to be the image of the generator $r^-$ of $K_1\big(M_4(C_0(R^-))\big)$ 
under the index map coming from the exact sequences 
\begin{equation*}
0 \longrightarrow M_4(C_0(T_{i,\omega(i)})) \longrightarrow 
M_4\big(C_0(\pi^{-1}(T_{i,\omega(i)}\cup R^{-}))\big)^\beta 
\longrightarrow M_4(C_0(R^{-})) \longrightarrow 0.
\end{equation*}
Since 
\[
M_4\big(C_0(\pi^{-1}(T_{2,3}\cup R^{-}))\big)^\beta 
\cong M_4\big(C_0(\hT_{2,3}^r \cup R_y^-)\big)
\cong M_4\big(C_0((0,1)\times (0,1])\big)
\]
whose K-groups are $0$, 
$w_{2,3}$ is a generator of $K_0\big(M_4(C_0(T_{2,3}))\big) \cong \Z$. 
Similarly, $w_{3,4}$ and $w_{4,2}$ are 
generators of $K_0\big(M_4(C_0(T_{3,4}))\big) \cong \Z$ 
and $K_0\big(M_4(C_0(T_{4,2}))\big) \cong \Z$, respectively. 

Similarly for $i=2,3,4$, 
we set the generator $w_{\omega(i),i}$ of 
$K_0\big(M_4(C_0(T_{\omega(i),i}))\big) \cong \Z$ 
to be the image of the generator $r^+$ of $K_1\big(M_4(C_0(R^+))\big)$ 
under the index map coming from the exact sequences 
\begin{equation*}
0 \longrightarrow M_4(C_0(T_{\omega(i)},i)) \longrightarrow 
M_4\big(C_0(\pi^{-1}(T_{\omega(i),i}\cup R^{+}))\big)^\beta 
\longrightarrow M_4(C_0(R^{+})) \longrightarrow 0.
\end{equation*}

Then the index map from 
\[
K_{1}(I^{\text{\textbullet}}) \cong K_1\big(M_4(C_0(R^-))\big)\oplus K_1\big(M_4(C_0(R^+)) \big) \cong \Z^2
\]
to 
\begin{align*}
K_0(I^\circ) \cong K_0\big(&M_4(C_0(T_{2,3}))\big) \oplus 
K_0\big(M_4(C_0(T_{3,4}))\big)\oplus 
K_0\big(M_4(C_0(T_{4,2}))\big)\\
&\oplus 
K_0\big(M_4(C_0(T_{3,2}))\big)\oplus 
K_0\big(M_4(C_0(T_{4,3}))\big)\oplus 
K_0\big(M_4(C_0(T_{2,4}))\big) \cong \Z^6
\end{align*}
becomes $\Z^2\ni (a,b) \mapsto (a,a,a,b,b,b) \in \Z^6$. 
Thus we have the following. 

\begin{proposition}
We have $K_0(I^{\star}) \cong \Z^4$ and $K_1(I^{\star})=0$. 
\end{proposition}

We denote by $s_1,s_2,s_3,s_4 \in K_0(I^{\star})$ the images of 
$w_{2,3},w_{3,4},w_{3,2},w_{4,3} \in K_0(I^\circ)$. 
Then $\{s_1,s_2,s_3,s_4\}$ becomes a basis of $K_0(I^{\star}) \cong \Z^4$. 
Note that the images of $w_{4,2},w_{2,4} \in K_0(I^\circ)$ are 
$-s_1-s_2\in K_0(I^{\star})$ and $-s_3-s_4\in K_0(I^{\star})$, respectively. 

We have a six-term exact sequence 
\begin{equation}
\xymatrix@C=12pt@R=20pt{
0=K_0(I_0) \ar[r] & K_0(I) \ar[r] & K_0(I^{\star}) \cong \Z^4 \ar[d] \\
0= K_{1}(I^{\star}) \ar[u] & \ar[l]  K_{1}(I) & \ar[l]  K_{1}(I_0) \cong \Z.
}
\end{equation}

To compute the index map $K_0(I^{\star}) \to K_{1}(I_0)$, 
we need the following lemma. 

\begin{lemma}
The index map from $K_0(I^\circ) \cong \Z^6$ to $K_{1}(I_0) \cong \Z$ 
coming from the short exact sequence 
\begin{equation*}
0 \longrightarrow I_0 \longrightarrow 
M_4\big(C_0(\pi^{-1}(O_0 \cup T))\big)^\beta \longrightarrow 
I^{\circ} \longrightarrow 0.
\end{equation*}
is $0$. 
\end{lemma}

\begin{proof}
We set $\hT \coloneqq \bigcup_{i,j} (\hT_{i,j}^r \cup \hT_{i,j}^l)$ 
where $i,j$ run $2,3,4$ with $i \neq j$. 
We have the following commutative diagram with exact rows; 
\begin{equation*}
\xymatrix@C=12pt@R=16pt{
0 \ar[r] & I_0 \ar[r]\ar[d]^{\cong} & 
M_4\big(C_0(\pi^{-1}(O_0 \cup T))\big)^\beta \ar[r]\ar[d] &
I^{\circ} \ar[r]\ar[d] & 0\phantom{.}\\
0 \ar[r] & M_4(C_0(V)) \ar[r] & 
M_4(C_0(V \cup \hT))\big) \ar[r] &
M_4(C_0(\hT)) \ar[r] & 0.
}
\end{equation*}
Note that $V \cup \hT = \pi^{-1}(O_0 \cup T)\cap \overline{V}$.
From this diagram, 
we see that the index map $K_0(I^\circ) \to K_{1}(I_0)$ 
factors through $K_0(M_4(C_0(\hT)))$. 

Take $i,j=2,3,4$ with $i\neq j$. 
Let $a_{i,j}^r \in K_0\big(M_4(C_0(\hT_{i,j}^r))\big)$ 
and $a_{i,j}^l \in K_0\big(M_4(C_0(\hT_{i,j}^l))\big)$ be 
the images of the generator $w_{i,j}$ of $K_0\big(M_4(C_0(T_{i,j}))\big)$ 
under the homomorphism induced by $\pi$. 
Under the map $K_0(I^{\circ}) \to K_0(M_4(C_0(\hT)))$, 
the generator $w_{i,j}$ of $K_0\big(M_4(C_0(T_{i,j}))\big)$ 
goes to $a_{i,j}^r + a_{i,j}^l$. 
Under the index map $K_0(M_4(C_0(\hT))) \to K_1\big(M_4(C_0(V))\big)$ 
the element $a_{i,j}^r + a_{i,j}^l$ goes to $0$ 
because the side to $V$ from $\hT_{i,j}^r$ and the one from $\hT_{i,j}^l$ 
differ if $\hT_{i,j}^r$ and $\hT_{i,j}^l$ are identified 
through the map $\pi$ to $T_{i,j}$. 
Thus we see that the map 
$K_0(I^\circ) \to K_1\big(M_4(C_0(V))\big) \cong K_{1}(I_0)$ is $0$. 
\end{proof}

By this lemma, 
the composition of the map $K_0(I^\circ) \to K_0(I^{\star})$ 
and the index map $K_0(I^{\star}) \to K_{1}(I_0)$ is $0$. 
Since the map $\Z^6 \cong K_0(I^\circ) \to K_0(I^{\star}) \cong \Z^4$ 
is a surjection, 
we see that the index map $K_0(I^{\star}) \to K_{1}(I_0)$ is $0$. 
Thus we have the following. 

\begin{proposition}\label{Prop:KI}
We have $K_0(I) \cong K_0(I^{\star}) \cong \Z^4$ and 
$K_1(I) \cong K_1(I_0) \cong \Z$. 
\end{proposition}

\section{K-groups of $A$}\label{Sec:KA}

Recall the six-term exact sequence 
\begin{equation*}
\xymatrix@C=25pt@R=20pt{
	K_0(I) \ar[r] &  K_0(A) \ar[r] & K_0(B) \ar[d]^{\delta_0} \\
	K_{1}(B) \ar[u]^{\delta_1} & \ar[l]   K_{1}(A) & \ar[l]  K_{1}(I).
}
\end{equation*}

In this section, we calculate the exponential map $\delta_0\colon K_0(B) \to K_1(I)$ and the index map $\delta_1\colon K_1(B) \to K_0(I)$.

\begin{proposition}
The exponential map $\delta_0\colon K_0(B) \to K_1(I)$ is $0$. 
\end{proposition}

\begin{proof}
Since $K_0(B)$ is generated 
by 16 elements $\{[P_{i,j}]_0\}_{i,j=1}^4$, 
the map $K_0(A) \to K_0(B)$ is surjective. 
Hence the exponential map $\delta_0\colon K_0(B) \to K_1(I)$ is $0$. 
\end{proof}

By the definitions of the generators of $K$-groups we did so far,
we have the following. 
(see Figure~\ref{fig:EF} for the relation between $T$ and $F$.) 

\begin{proposition}
The index map 
$\delta''\colon K_1(B^{\circ})\cong \Z^{18} \to K_0(I^\circ) \cong \Z^{6}$ 
coming from the short exact sequence
\begin{equation*}
0 \longrightarrow I^\circ \longrightarrow 
M_4\big(C_0(\pi^{-1}(T \cup F^{\circ}))\big)^\beta \longrightarrow 
B^{\circ} \longrightarrow 0.
\end{equation*}
is as Table \ref{table:ind_def}.
\end{proposition}

\begin{table}[hbtp]
%\begin{center}
\caption{Computation of the index map $\delta''$}\label{table:ind_def}
\begin{tabular}{|c||cc|cc|cc|cc|cc|cc|cc|cc|cc|} \hline
& \multicolumn{2}{|c|}{2,2} &  \multicolumn{2}{|c|}{3,3} &  \multicolumn{2}{|c|}{4,4} &  \multicolumn{2}{|c|}{2,3} &  \multicolumn{2}{|c|}{3,4} &  \multicolumn{2}{|c|}{4,2} &  \multicolumn{2}{|c|}{3,2} &  \multicolumn{2}{|c|}{4,3} &  \multicolumn{2}{|c|}{2,4} \\ \hline
$w\diagdown v$ & $\cap$ & $\cup$ & $\cap$ & $\cup$ & $\cap$ & $\cup$ & $\cap$ & $\cup$ & $\cap$ & $\cup$ & $\cap$ & $\cup$ & $\cap$ & $\cup$ & $\cap$ & $\cup$ & $\cap$ & $\cup$  \\ \hline\hline
2,3 & 0 & 0 & 0 & 0 &\mo &\mo & 1 & 1 & 0 & 0 & 0 & 0 & 0 & 0 & 0 & 0 & 0 & 0  \\ \hline
3,4 &\mo &\mo & 0 & 0 & 0 & 0 & 0 & 0 & 1 & 1 & 0 & 0 & 0 & 0 & 0 & 0 & 0 & 0  \\ \hline
4,2 & 0 & 0 &\mo &\mo & 0 & 0 & 0 & 0 & 0 & 0 & 1 & 1 & 0 & 0 & 0 & 0 & 0 & 0  \\ \hline
3,2 & 0 & 0 & 0 & 0 &\mo &\mo & 0 & 0 & 0 & 0 & 0 & 0 & 1 & 1 & 0 & 0 & 0 & 0  \\ \hline
4,3 &\mo &\mo & 0 & 0 & 0 & 0 & 0 & 0 & 0 & 0 & 0 & 0 & 0 & 0 & 1 & 1 & 0 & 0  \\ \hline
2,4 & 0 & 0 &\mo &\mo & 0 & 0 & 0 & 0 & 0 & 0 & 0 & 0 & 0 & 0 & 0 & 0 & 1 & 1  \\ \hline
\end{tabular}
\end{table}

\begin{definition}
The composition of the index map 
$\delta''\colon K_1(B^{\circ}) \to K_0(I^\circ)$ 
and the map $K_0(I^\circ) \to K_0(I^\star)$ is 
denoted by $\eta\colon K_1(B^{\circ}) \to K_0(I^\star)$ 

We set $\widetilde{\eta}\colon K_1(B^{\circ}) \to K_0(I^\star) \oplus \Z/2\Z$ 
by $\widetilde{\eta}(w_{i,j}^\cap)=(\eta(w_{i,j}^\cap),0)$ and 
$\widetilde{\eta}(w_{i,j}^\cup)=(\eta(w_{i,j}^\cup),1)$ for $i,j=2,3,4$. 
\end{definition}

We denote the generator of $\Z/2\Z$ in $K_0(I^\star) \oplus \Z/2\Z$ 
by $s_5$. 

\begin{table}[hbtp]
%\begin{center}
\caption{Computation of $\widetilde{\eta}$}
\begin{tabular}{|c||cc|cc|cc|cc|cc|cc|cc|cc|cc|} \hline
& \multicolumn{2}{|c|}{2,2} &  \multicolumn{2}{|c|}{3,3} &  \multicolumn{2}{|c|}{4,4} &  \multicolumn{2}{|c|}{2,3} &  \multicolumn{2}{|c|}{3,4} &  \multicolumn{2}{|c|}{4,2} &  \multicolumn{2}{|c|}{3,2} &  \multicolumn{2}{|c|}{4,3} &  \multicolumn{2}{|c|}{2,4} \\ \hline
$s\diagdown v$ & $\cap$ & $\cup$ & $\cap$ & $\cup$ & $\cap$ & $\cup$ & $\cap$ & $\cup$ & $\cap$ & $\cup$ & $\cap$ & $\cup$ & $\cap$ & $\cup$ & $\cap$ & $\cup$ & $\cap$ & $\cup$  \\ \hline\hline
1 & 0 & 0 & 1 & 1 &\mo &\mo & 1 & 1 & 0 & 0 &\mo &\mo & 0 & 0 & 0 & 0 & 0 & 0  \\ \hline
2 &\mo &\mo & 1 & 1 & 0 & 0 & 0 & 0 & 1 & 1 &\mo &\mo & 0 & 0 & 0 & 0 & 0 & 0  \\ \hline
3 & 0 & 0 & 1 & 1 &\mo &\mo & 0 & 0 & 0 & 0 & 0 & 0 & 1 & 1 & 0 & 0 &\mo &\mo  \\ \hline
4 &\mo &\mo & 1 & 1 & 0 & 0 & 0 & 0 & 0 & 0 & 0 & 0 & 0 & 0 & 1 & 1 &\mo &\mo  \\ \hline
5 & 0 & 1 & 0 & 1 & 0 & 1 & 0 & 1 & 0 & 1 & 0 & 1 & 0 & 1 & 0 & 1 & 0 & 1  \\ \hline
\end{tabular}
\end{table}

\begin{proposition}\label{Prop:kernel=image}
The map $\widetilde{\eta}\colon K_1(B^{\circ}) \to K_0(I^\star) \oplus \Z/2\Z$ 
is surjective, and its kernel coincides with the image of 
$\delta\colon K_0(B^{\text{\textbullet}}) \to K_1(B^{\circ})$. 
\end{proposition}

\begin{proof}
Since 
\begin{align*}
\widetilde{\eta}(w_{2,3}^\cap)=s_1, \quad 
\widetilde{\eta}(w_{3,4}^\cap)=s_2, \quad
\widetilde{\eta}(w_{3,2}^\cap)=s_3, \quad
\widetilde{\eta}(w_{4,3}^\cap)=s_4, 
\end{align*}
$s_1,s_2,s_3,s_4$ are in the image of $\widetilde{\eta}$. 
Since $\widetilde{\eta}(w_{2,2}^\cup+w_{3,3}^\cup+w_{4,4}^\cup) = s_5$, 
$s_5$ is also in the image of $\widetilde{\eta}$. 
Thus $\widetilde{\eta}$ is surjective. 

It is straightforward to check $\widetilde{\eta} \circ \delta = 0$
Hence the image of $\delta$ is 
contained in the kernel of $\widetilde{\eta}$. 
Suppose 
\[
x= \sum_{i,j=2}^4 n_{i,j}^{\cap}w_{i,j}^{\cap}
+\sum_{i,j=2}^4 n_{i,j}^{\cup}w_{i,j}^{\cup}
\]
is in the kernel of $\widetilde{\eta}$ 
where $n_{i,j}^{\cap},n_{i,j}^{\cup} \in \Z$ for $i,j=2,3,4$. 
We will show that $x$ is in the image of $\delta$. 
By adding 
\[
n_{2,3}^{\cup}\delta(q_{(3142)})+
n_{3,4}^{\cup}\delta(q_{(4312)})+
n_{4,2}^{\cup}\delta(q_{(2341)})+
n_{3,2}^{\cup}\delta(q_{(2413)})+
n_{4,3}^{\cup}\delta(q_{(3421)})+
n_{2,4}^{\cup}\delta(q_{(4123)})
\]
we may assume 
\[
n_{2,3}^{\cup}=n_{3,4}^{\cup}=n_{4,2}^{\cup}=n_{3,2}^{\cup}=n_{4,3}^{\cup}
=n_{2,4}^{\cup}=0
\]
without loss of generality. 
By subtracting $n_{3,3}^{\cup}\delta(q_{(4321)})
+n_{4,4}^{\cup}\delta(q_{(3412)})$, 
we may further assume $n_{3,3}^{\cup}=n_{4,4}^{\cup}=0$ 
without loss of generality. 
Then $n_{2,2}^{\cup}$ is even since 
the coefficient of $c_5$ for $\widetilde{\eta}(x)$ is $0$. 
Hence by adding 
\[
\frac{n_{2,2}^{\cup}}{2}
\big(\delta(q_{(2143)})-\delta(q_{(3412)})-\delta(q_{(4321)})\big)
\]
we may further assume $n_{2,2}^{\cup}=0$ 
without loss of generality. 
Thus we may assume 
$x = \sum_{i,j=2}^4 n_{i,j}^{\cap}w_{i,j}^{\cap}$. 
By adding $n_{2,2}^{\cap}\delta(q_{(1243)})+n_{3,3}^{\cap}\delta(q_{(1432)})+n_{4,4}^{\cap}\delta(q_{(1324)})$, 
we may further assume $n_{2,2}^{\cap}=n_{3,3}^{\cap}=n_{4,4}^{\cap}=0$
without loss of generality. 
By subtracting 
$n_{4,2}^{\cap}\delta(q_{(1423)})+n_{2,4}^{\cap}\delta(q_{(1342)})$, 
we may further assume $n_{4,2}^{\cap}=n_{2,4}^{\cap}=0$
without loss of generality. 
Thus we may assume 
\[
x=n_{2,3}^{\cap}w_{2,3}^{\cap}+n_{3,4}^{\cap}w_{3,4}^{\cap}
+n_{3,2}^{\cap}w_{3,2}^{\cap}+n_{4,3}^{\cap}w_{4,3}^{\cap}. 
\]
Then we have $n_{2,3}^{\cap}=n_{3,4}^{\cap}=n_{3,2}^{\cap}=n_{4,3}^{\cap}=0$ 
because 
\[
\widetilde{\eta}(x)
=n_{2,3}^{\cap}s_1+n_{3,4}^{\cap}s_2
+n_{3,2}^{\cap}s_3+n_{4,3}^{\cap}s_4. 
\]
Thus $x=0$. 
We have shown that $x$ is in the image of $\delta$. 
Hence the image of $\delta$ coincides with the kernel of $\widetilde{\eta}$. 
\end{proof}

As a corollary of this proposition, 
we have the following as predicted. 

\begin{proposition}\label{Prop:KB}
We have $K_0(B) \cong \Z^{10}$ 
and $K_1(B) \cong \Z^4 \oplus \Z/2\Z$. 
\end{proposition}

\begin{proof}
By Proposition~\ref{Prop:kernel=image}, 
we see that $K_1(B) \cong \coker \delta$ is isomorphic 
to $\Z^4 \oplus \Z/2\Z$. 
This implies $K_0(B) \cong \ker \delta$ is isomorphic to $\Z^{10}$ 
because $\ker \delta$ is a free abelian group with dimension $24-18+4=10$. 
\end{proof}

We also have the following. 

\begin{proposition}
The index map $\delta_1\colon K_1(B) \to K_0(I)$ is as $K_1(B) \cong \Z^4 \oplus \Z/ 2\Z \ni (n,m) \mapsto n \in \Z^4 \cong K_0(I)$. 
\end{proposition}

\begin{proof}
From the commutative diagram with exact rows
\begin{equation*}
\xymatrix@C=12pt@R=16pt{
0 \ar[r] & I \ar[r]\ar[d] & 
A \ar[r]\ar[d] &
B \ar[r]\ar@{=}[d] & 0\phantom{.}\\
0 \ar[r] & I^{\star} \ar[r] & 
M_4\big(C_0(\pi^{-1}((O\setminus O_0) \cup F))\big)^\beta \ar[r] &
B \ar[r] & 0,
}
\end{equation*}
the index map $\delta_1\colon K_1(B) \to K_0(I)$ coincides with 
the map $K_1(B) \to K_0(I^{\star})$ 
if we identify $K_0(I) \cong K_0(I^{\star})$ 
as we did in Proposition~\ref{Prop:KI}. 

From the commutative diagram with exact rows
\begin{equation*}
\xymatrix@C=12pt@R=16pt{
0 \ar[r] & I^{\circ} \ar[r]\ar[d] & 
M_4\big(C_0(\pi^{-1}(T \cup F^{\circ}))\big)^\beta \ar[r]\ar[d] &
B^{\circ} \ar[r]\ar[d] & 0\phantom{.}\\
0 \ar[r] & I^{\star} \ar[r] & 
M_4\big(C_0(\pi^{-1}((O\setminus O_0) \cup F))\big)^\beta \ar[r] &
B \ar[r] & 0,
}
\end{equation*}
we have the commutative diagram
\begin{align*}
\xymatrix@C=25pt@R=20pt{
K_1(B^\circ) \ar[d] \ar[r] & K_0(I^{\circ}) \ar[d] \\
K_1(B) \ar[r] & K_{0}(I^{\star}) .
}
\end{align*}
From this diagram, we see that the map $K_1(B) \to K_0(I^{\star})$ is 
as $K_1(B) \cong \Z^4 \oplus \Z/ 2\Z \ni (n,m) \mapsto 
n \in \Z^4 \cong K_0(I^{\star})$. 
This completes the proof. 
\end{proof}

\begin{definition}
Define a unitary $w \in C(S^3,M_2(\C))$ by 
\begin{align*}
w(a_1,a_2,a_3,a_4) &= a_1 c_1 +a_2 c_2 +a_3 c_3 +a_4 c_4 \\
&= \begin{pmatrix}
a_1 +a_2 \sqrt{-1} & a_3 +a_4 \sqrt{-1} \\
-a_3 + a_4 \sqrt{-1} & a_1 - a_2\sqrt{-1}
\end{pmatrix}
\end{align*}
for $(a_1,a_2,a_3,a_4) \in S^3$. 
\end{definition}

Then $[w]_1$ is the generator of 
$K_1\big(C(S^3,M_2(\C))\big) \cong K_1\big(M_4(C(S^3))\big) \cong \Z$. 

Let $\varphi \colon A \to M_4(C(S^3))$ be the composition of 
the embedding $A \to M_4(C(\R P^3))$ and 
the map $M_4(C(\R P^3)) \to M_4(C(S^3))$ 
induced by $[\cdot]\colon S^3 \to \R P^3$. 
Let $\widetilde{\pi}\colon S^3 \to X$ be the composition of 
$[\cdot]\colon S^3 \to \R P^3$ and $\pi\colon \R P^3 \to X$. 
We set $V'$ of $S^3$ by 
\begin{equation*}
V'\coloneqq \{ (a_1,a_2, a_3, a_4) \in S^3 \mid a_1,a_2,a_3 > |a_4| \}.
\end{equation*}
Then $V'$ is homeomorphic to $V$ via $[\cdot]$, 
and hence to $O_0$ via $\widetilde{\pi}$. 
Note that the map 
$M_4(C_0(V')) \hookrightarrow M_4(C(S^3))$ 
induces the isomorphism 
\[
K_1\big(M_4(C_0(V'))\big) \to K_1\big(M_4(C(S^3))\big). 
\]
Since $I_0 \cong M_4(C_0(O_0)) \cong M_4(C_0(V'))$ canonically, 
we set a generator $y$ of $K_1(I_0)$ which corresponds 
to the generator $[w]_1$ of $K_1\big(M_4(C(S^3))\big)$ 
via the isomorphism $K_1\big(M_4(C_0(V'))\big) \to K_1\big(M_4(C(S^3))\big)$. 
We denote by the same symbol $y$ the generator of 
$K_1(I) \cong K_1(I_0)$ corresponding $y \in K_1(I_0)$. 

\begin{proposition}\label{Prop:32}
The image of $y \in K_1(I)$ 
under the map $K_1(I) \to K_1(A) \to K_1\big(M_4(C(S^3))\big)$ 
is $32[w]_1$. 
\end{proposition}

\begin{proof}
The map $I_0 \to I \to A \to M_4(C(S^3))$ 
is induced by $\widetilde{\pi}\colon \widetilde{\pi}^{-1}(O_0) \to O_0$ 
when we identify $I_0$ with $M_4(C_0(O_0))$. 
We have 
\[
\widetilde{\pi}^{-1}(O_0) = 
\coprod_{i,j=1}^4 \sigma'_{i,j,+}(V') 
\amalg \coprod_{i,j=1}^4 \sigma'_{i,j,-}(V')
\]
where $\sigma'_{i,j,\pm}\colon S^3 \to S^3$ is induced 
by the unitary $\pm U_{i,j}$ 
similarly as $\sigma_{i,j}\colon \R P^3 \to \R P^3$ 
for $i,j=1,2,3,4$. 
These $32$ homeomorphisms preserve the orientation of $S^3$. 
Therefore, the image of $y \in K_1(I_0)$, and hence the one of $y \in K_1(I)$, 
in $K_1\big(M_4(C(S^3))\big)$ is $32[w]_1$. 
\end{proof}

\begin{definition}
Define the linear map $\xi : M_2(\C) \to \C^4$ by 
\begin{align*}
&\xi \left(
\begin{pmatrix}
a_{11} & a_{12} \\
a_{21} & a_{22}
\end{pmatrix}
\right)= \frac{1}{\sqrt{2}}(a_{11}, a_{12}, a_{21}, a_{22}).
\end{align*}
\end{definition}

\begin{definition}
Define unital $*$-homomorphisms $\iota, \iota'\colon M_2(\C) \to M_4(\C)$ by 
\begin{align*}
\iota \left(  
\begin{pmatrix}
a_{11} & a_{12}  \\
a_{21} & a_{22}
\end{pmatrix}
\right)=
\begin{pmatrix}
a_ {11}  & a_{12} & 0 & 0 \\
a_{21} & a_{22} & 0 & 0 \\
0  & 0 & a_{11} & a_{21} \\
0 & 0 & a_{21} & a_{22}  
\end{pmatrix}, \\
\iota' \left(  
\begin{pmatrix}
a_{11} & a_{12}  \\
a_{21} & a_{22}
\end{pmatrix}
\right)=
\begin{pmatrix}
a_ {11}  & 0 & a_{12} & 0 \\
0 & a_{11} & 0 & a_{12}  \\
a_ {21}  & 0 & a_{22} & 0 \\
0 & a_{21} & 0 & a_{22}  
\end{pmatrix}.
\end{align*}
\end{definition}

\begin{lemma}\label{Lem:xi}
For each $M, N \in M_2(\C)$, we have
\begin{align*}
\xi(M) \iota(N) &= \xi(MN),&
\iota'(M) \xi(N)^{\mathrm{T}} &= \xi(MN)^{\mathrm{T}}.
\end{align*}
\end{lemma}

\begin{proof}
It follows from a direct computation. 
\end{proof}

\begin{definition}
Define $U \in M_4(A)$ by
\begin{equation*}
U=
\begin{pmatrix}
P_{11} & P_{12} & P_{13} & P_{14} \\
P_{21} & P_{22} & P_{23} & P_{24} \\
P_{31} & P_{32} & P_{33} & P_{34} \\
P_{41} & P_{42} & P_{43} & P_{44} 
\end{pmatrix}.
\end{equation*}
\end{definition}

It can be easily checked that $U$ is a unitary. 

\begin{proposition}\label{Prop:16}
The image of $[U]_1 \in K_1(A)$ 
under the map $K_1(A) \to K_1\big(M_4(C(S^3))\big)$ 
is $16[w]_1$. 
\end{proposition}

\begin{proof}
Let $\varphi_4\colon M_4(A) \to M_4\big(M_4(C(S^3))\big)$ 
be the \shom induced by $\varphi$. 
Set $\bU \coloneqq \varphi_4(U)$. 
For $i,j=1,2,3,4$, 
the $(i,j)$-entry $\bU_{i,j} \in C(S^3, M_4(\C))$ of $\bU$ 
is given by 
\begin{equation*}
\bU_{i,j}(a_1,a_2,a_3,a_4)= U_{i,j} (a_1,a_2,a_3,a_4)^\mathrm{T} (a_1,a_2,a_3,a_4) U_{i,j}^*
\end{equation*}
for each $(a_1,a_2,a_3,a_4) \in S^3$.

Let $W \in M_4(\C)$ be 
\begin{equation*}
W=\frac{1}{\sqrt{2}}
\begin{pmatrix}
1 & -\sqrt{-1} & 0 & 0 \\
0 & 0 & 1 & -\sqrt{-1} \\
0 & 0 & -1 &  -\sqrt{-1} \\
1 & \sqrt{-1} & 0 & 0 
\end{pmatrix}.
\end{equation*}
Then $W$ is a unitary. 

Take $(a_1,a_2,a_3,a_4) \in S^3$ and $i,j= 1,2,3,4$. 
We set $(b_1,b_2,b_3,b_4) = (a_1,a_2,a_3,a_4) U_{i,j}^*$. 
By Proposition~\ref{Prop:Udef}, we have 
$\sum_{k=1}^4 b_k c_k=c_i \Big(\sum_{k=1}^4 a_k c_k \Big)c_j^*$. 
We also have 
\begin{align*}
\xi\Big(\sum_{k=1}^4 b_k c_k\Big)W 
&=\frac{1}{\sqrt{2}}(b_1+b_2\sqrt{-1}, b_3+b_4\sqrt{-1}, 
-b_3+b_4\sqrt{-1},b_1-b_2\sqrt{-1})W\\
&=(b_1,b_2,b_3,b_4)
\end{align*}
Hence we get 
\begin{align*}
(a_1,a_2,a_3,a_4) U_{i,j}^*
&=\xi\left(c_i \Big(\sum_{k=1}^4 a_k c_k \Big)c_j^*\right)W\\
&=\xi(c_i)\iota\left(\Big(\sum_{k=1}^4 a_k c_k \Big)c_j^*\right)W\\
&=\xi(c_i)\iota(w(a_1,a_2,a_3,a_4))\iota(c_j^*)W
\end{align*}
by Lemma~\ref{Lem:xi}. 
Similarly, we get
\begin{align*}
U_{i,j} (a_1,a_2,a_3,a_4)^{\mathrm{T}} 
&= W^{\mathrm{T}} \xi\left(c_i \Big(\sum_{k=1}^4 a_k c_k \Big)c_j^*\right)^{\mathrm{T}} \\
&= W^{\mathrm{T}}\iota'\left(c_i \Big(\sum_{k=1}^4 a_k c_k \Big)\right)
\xi(c_j^*)^{\mathrm{T}} \\
&= W^{\mathrm{T}}\iota'(c_i)\iota'(w(a_1,a_2,a_3,a_4))\xi(c_j^*)^{\mathrm{T}}
\end{align*}
by Lemma~\ref{Lem:xi}. 
Define $\mathbb{V}, \mathbb{W}, \mathbb{W}' \in M_4(M_{4}(\C))$ by
\begin{align*}
&\mathbb{V}=(\xi(c_j^*)^{\mathrm{T}} \xi(c_i))_{i,j=1}^4,\\
&\mathbb{W}=
\begin{pmatrix}
\iota(c_1^*) W & 0 & 0 & 0 \\
0 & \iota(c_2^*) W & 0 & 0 \\
0 & 0 & \iota(c_3^*) W & 0 \\
0 & 0  & 0 & \iota(c_4^*) W
\end{pmatrix},\\
&\mathbb{W}'= 
\begin{pmatrix}
W^{\mathrm{T}} \iota' (c_1) &  0 & 0 & 0 \\
0 & W^{\mathrm{T}} \iota' (c_2) & 0 & 0 \\
0 & 0 &   W^{\mathrm{T}} \iota' (c_3) & 0 \\
0 & 0 & 0 &  W^{\mathrm{T}} \iota' (c_4)
\end{pmatrix}. 
\end{align*}
One can check that these are unitaries. 
If we consider these as constant functions in $M_4\big(C(S^3,M_4(\C))\big)$, 
we have 
\begin{equation*}
\bU= \mathbb{W}' \iota'_4(w) \mathbb{V} \iota_4(w) \mathbb{W},
\end{equation*}
where $\iota_4(w), \iota'_4(w) \in M_4(C(S^3,M_4(\C))$ are defined as 
\begin{align*}
&\iota_4(w)= 
\begin{pmatrix}
\iota(w(\cdot)) & 0 & 0 & 0 \\
0 & \iota (w(\cdot)) & 0 & 0 \\
0 & 0 & \iota (w(\cdot)) & 0 \\
0 & 0 & 0 &  \iota (w(\cdot))
\end{pmatrix}, \\
&\iota'_4(w)=
\begin{pmatrix}
\iota' (w(\cdot)) & 0 & 0 & 0 \\
0 & \iota' (w(\cdot)) & 0 & 0 \\
0 & 0 & \iota' (w(\cdot)) & 0 \\
0 & 0 & 0 &  \iota' (w(\cdot))
\end{pmatrix}.
\end{align*}
Since $[\iota_4(w)]_1= [\iota'_4(w)]_1=8[w]_1$, 
we obtain $[\bU]_1=16[w]_1$.
\end{proof}

\begin{proposition}\label{Prop:K(A)}
We have $K_0(A) \cong \Z^{10}$ and $K_1(A) \cong \Z$.
More specifically, $K_0(A)$ is generated by $\{[P_{i,j}]_0\}_{i,j=1}^4$, 
and $K_1(A)$ is generated by $[U]_1$.
Moreover, the positive cone $K_0(A)_+$ of $K_0(A)$ is 
generated by $\{[P_{i,j}]_0\}_{i,j=1}^4$ as a monoid. 
\end{proposition}

\begin{proof}
We have already seen that $K_0(A) \to K_0(B)$ is isomorphic, 
and we have a short exact sequence
\begin{equation*}
0 \longrightarrow K_1(I) \longrightarrow K_1(A) \longrightarrow \Z/ 2 \Z \longrightarrow 0.
\end{equation*}
From this, we see that $K_1(A)$ is isomorphic to either $\Z \oplus \Z/ 2 \Z$ 
or $\Z$. 
If $K_1(A)$ is isomorphic to $\Z \oplus \Z/ 2 \Z$, 
one can choose an isomorphism so that 
$y \in K_1(I)$ goes to $(1,0) \in \Z \oplus \Z/ 2 \Z$. 
Then the image of the map $K_1(A) \to K_1\big(M_4(C(S^3))\big) \cong \Z$ 
is $32\Z$ by Proposition~\ref{Prop:32}. 
This is a contradiction because the image of $[U]_1 \in K_1(A)$ is $16$ 
by Proposition~\ref{Prop:16}. 
Hence $K_1(A)$ is isomorphic to $\Z$ so that 
$y \in K_1(I)$ goes to $2$. 
By Proposition~\ref{Prop:32} and Proposition~\ref{Prop:16}, 
$[U]_1 \in K_1(A)$ corresponds to $1 \in \Z$. 
Thus $[U]_1$ is a generator of $K_1(A) \cong \Z$. 

It is clear that the monoid generated by $\{[P_{i,j}]_0\}_{i,j=1}^4$ 
is contained in the positive cone $K_0(A)_+$. 
The positive cone $K_0(A)_+$ maps 
into the positive cone $K_0(B^{\text{\textbullet}})_+$ 
under the surjection $A \to B^{\text{\textbullet}}$. 
Hence by Proposition~\ref{Prop:poscone}, 
$K_0(A)_+$ is contained in 
the monoid generated by $\{[P_{i,j}]_0\}_{i,j=1}^4$. 
Thus $K_0(A)_+$ is 
the monoid generated by $\{[P_{i,j}]_0\}_{i,j=1}^4$. 
\end{proof}

\begin{definition}\label{Def:defuni}
Define $u \in M_4(A(4))$ by
\begin{equation*}
u=
\begin{pmatrix}
p_{11} & p_{12} & p_{13} & p_{14} \\
p_{21} & p_{22} & p_{23} & p_{24} \\
p_{31} & p_{32} & p_{33} & p_{34} \\
p_{41} & p_{42} & p_{43} & p_{44} 
\end{pmatrix}.
\end{equation*}
\end{definition}

It can be easily checked that $u$ is a unitary. 
This unitary $u$ is called the {\em defining unitary} of 
the magic square C*-algebra $A(4)$. 

By Proposition~\ref{Prop:K(A)}, 
we get the third main theorem. 

\begin{theorem}\label{MainThm3}
We have $K_0(A(4)) \cong \Z^{10}$ and $K_1(A(4)) \cong \Z$.
More specifically, $K_0(A(4))$ is generated by $\{[p_{i,j}]_0\}_{i,j=1}^4$, 
and $K_1(A(4))$ is generated by $[u]_1$.

The positive cone $K_0(A(4))_+$ of $K_0(A(4))$ is generated by $\{[p_{i,j}]_0\}_{i,j=1}^4$ as a monoid. 
\end{theorem}

As mentioned in the introduction, 
the computation $K_0(A(4)) \cong \Z^{10}$ and $K_1(A(4)) \cong \Z$ 
and that $K_0(A(4))$ is generated by $\{[p_{i,j}]_0\}_{i,j=1}^4$
were already obtained by Voigt in \cite{V}. 
We give totally different proofs of these facts. 
That $K_1(A(4))$ is generated by $[u]_1$ and the computation 
of the positive cone $K_0(A(4))_+$ of $K_0(A(4))$ are new.

\end{document}